\documentclass{amsart}
\usepackage{amsfonts}
\usepackage{amsmath}
\usepackage{amssymb}
\usepackage{amsthm}
\usepackage{graphicx}
\usepackage{color}
\usepackage{mathtools}
\usepackage[utf8]{inputenc}
\usepackage[T1]{fontenc}
\usepackage[english]{babel}
\newtheorem{theorem}{Theorem}[section]
\newtheorem{proposition}[theorem]{Proposition}
\newtheorem{lemma}[theorem]{Lemma}
\newtheorem{corollary}[theorem]{Corollary}
\theoremstyle{definition}
\newtheorem{definition}[theorem]{Definition}
\newtheorem{example}[theorem]{Example}
\newtheorem{remark}[theorem]{Remark}
\numberwithin{equation}{section}

\newcommand*\diff{\mathop{}\!\mathrm{d}}

\DeclareMathOperator{\id}{Id}

\DeclareMathOperator{\Tr}{Tr}
\begin{document}
\title[Decay rates for linear systems with random switching]{Decay rates for stabilization of linear continuous-time systems with random switching}
\author[Fritz Colonius]{Fritz Colonius\textsuperscript{$*$}}
\author[Guilherme Mazanti]{Guilherme Mazanti\textsuperscript{$\dagger$}}
\thanks{\textsuperscript{$*$}Institut f\"{u}r Mathematik, Universit\"{a}t Augsburg,
86159 Augsburg, Germany. Supported by DFG grant Co 124/19-1.}
\thanks{\textsuperscript{$\dagger$}Laboratoire de Math\'ematiques d'Orsay,
Univ.\ Paris-Sud, Universit\'e Paris-Saclay, 91405 Orsay, France. This work
was prepared while the author was with CMAP \& Inria, team GECO,
\'Ecole Polytechnique, CNRS, Universit\'e Paris-Saclay, 91128 Palaiseau Cedex, France.
Partially supported by the iCODE Institute, research project of the IDEX
Paris-Saclay, and by the Hadamard Mathematics LabEx (LMH) through the grant
number ANR-11-LABX-0056-LMH in the ``Programme des Investissements d'Avenir''.}
\keywords{Random switching, almost sure stabilization, arbitrary rate of convergence,
Lyapunov exponents, persistent excitation, Multiplicative Ergodic Theorem}
\subjclass[2010]{93C30, 93D15, 37H15}
\date{\today}

\begin{abstract}
For a class of linear switched systems in continuous time a controllability
condition implies that state feedbacks allow to achieve almost sure
stabilization with arbitrary exponential decay rates. This is based on the
Multiplicative Ergodic Theorem applied to an associated system in discrete
time. This result is related to the stabilizability problem for linear persistently excited systems.

\end{abstract}
\maketitle

\renewcommand{\theenumi}{(\roman{enumi})} \renewcommand{\labelenumi}{\theenumi}

\section{Introduction}

Let $N$ be a positive integer and consider the family of $N$ control systems
\begin{equation}
\label{IntroMainSyst}\dot{x}_{i}(t) = A_{i}x_{i}(t) + \alpha_{i}(t) B_{i}
u_{i}(t), \qquad i \in\{1, \dotsc, N\},
\end{equation}
where, for $i \in\{1, \dotsc, N\}$, $x_{i}(t) \in\mathbb{R}^{d_{i}}$ is the
state of the subsystem $i$, $u_{i}(t)\in\mathbb{R}^{m_{i}}$ is the control
input of the subsystem $i$, $d_{i}$ and $m_{i}$ are non-negative integers,
$A_{i}$ and $B_{i}$ are matrices with real entries and appropriate dimensions,
and $\alpha_{i}:\mathbb{R}_{+}\rightarrow\{0,1\}$ is a switching signal
determining the activity of the control input on the $i$-th subsystem. We
assume that at each time the control input is active in exactly one subsystem,
i.e.,
\begin{equation}
\label{HypoOnlyOneAlpha}\sum\nolimits_{i=1}^{N} \alpha_{i}(t) = 1 \text{ for
all } t \in\mathbb{R}_{+}.
\end{equation}
This paper analyzes the stabilizability of all subsystems in
\eqref{IntroMainSyst} by linear feedback laws $u_{i}(t)=K_{i}x_{i}(t)$ under
randomly generated switching signals $\alpha_{1},\dotsc,\alpha_{N}$ satisfying
\eqref{HypoOnlyOneAlpha}, and the maximal almost sure exponential decay rates
that can be achieved with such feedbacks.

System \eqref{IntroMainSyst} is a switched control system, where the switching
signals $\alpha_{1},\dotsc,\alpha_{N}$ affect the activity of the control
input. Switched systems have been extensively studied in the literature, both
for deterministic switching signals, such as in the monographs Liberzon
\cite{Liberzon2003Switching} and Sun and Ge \cite{Sun2005Switched} and the
surveys Lin and Antsaklis \cite{Lin2009Stability}, Margaliot
\cite{Margaliot2006Stability}, and Shorten, Wirth, Mason, Wulff, and King
\cite{Shorten2007Stability}, and for random switching signals, such as in the
monographs Costa, Fragoso, and Todorov \cite{Costa2013Continuous} and Davis
\cite{Davis1993Markov}, and papers such as Bena\"{\i}m, Le Borgne, Malrieu,
and Zitt \cite{BBMZ15}, Cloez and Hairer \cite{Cloez2015Exponential}, and
Guyon, Iovleff, and Yao \cite{Guyon2004Linear}. Such systems are useful models
in several applications, ranging from air traffic control, electronic
circuits, and automotive engines to chemical processes and population models
in biology.

An important motivation for our work comes from the theory of persistently
excited control systems, in which one considers systems of the form
\begin{equation}
\label{IntroPESyst}\dot{x}(t)=Ax(t)+\alpha(t)Bu(t),
\end{equation}
with $x(t)\in\mathbb{R}^{d}$, $u(t)\in\mathbb{R}^{m}$, $A$ and $B$ matrices
with real entries and appropriate dimensions, and $\alpha$ a $(T,\mu
)$-persistently exciting (PE) signal for some positive constants $T\geq\mu$,
i.e., a signal $\alpha\in L^{\infty}(\mathbb{R}_{+},[0,1])$ satisfying, for
every $t\geq0$,
\begin{equation}
\label{CondPE}\int_{t}^{t+T}\alpha(s)\diff s\geq\mu
\end{equation}
(cf.\ Chaillet, Chitour, Lor\'{\i}a, and Sigalotti \cite{Chaillet2008Uniform},
Chitour, Colonius, and Sigalotti \cite{Chitour2014Growth}, Chitour, Mazanti,
and Sigalotti \cite{Chitour2013Stabilization}, Chitour and Sigalotti
\cite{Chitour2010Stabilization}, Srikant and Akella
\cite{Srikant2015Arbitrarily}). Notice that, when $\alpha$ takes its values in
$\{0,1\}$, \eqref{IntroPESyst} can be seen as a particular case of
\eqref{IntroMainSyst} by adding a trivial subsystem (cf.\ Corollary \ref{CoroToMainTheo}). The stabilizability problem for \eqref{IntroPESyst}
consists in investigating if, given $A$, $B$, $T$, and $\mu$, one can find a
linear feedback $u(t)=Kx(t)$ which stabilizes \eqref{IntroPESyst}
exponentially for every $(T,\mu)$-persistently exciting signal $\alpha$. This
problem has been considered in \cite{Chitour2010Stabilization}, where the
authors provide sufficient conditions for stabilizability and prove that, in
contrast to the situation for autonomous linear control systems,
controllability does not imply stabilizability with arbitrary exponential
decay rates, even if one considers only persistently exciting signals taking
values in $\{0,1\}$. The main result of our paper, Theorem \ref{MainTheoApp},
implies that, if one requires the feedback to stabilize \eqref{IntroPESyst}
for \emph{almost every} randomly generated signal $\alpha$ (with respect to
the random model described in Section \ref{SecMarkov}), then one can retrieve
stabilizability with arbitrary decay rates, giving thus a positive answer to
an open problem stated by Chitour and Sigalotti (personal communication).

Some works in the literature have addressed the stabilization of systems
similar to \eqref{IntroPESyst} with randomly generated signals $\alpha$, such
as Diwadkar, Dasgupta, and Vaidya \cite{Diwadkar2015Control} and Diwadkar and
Vaidya \cite{Diwadkar2014Stabilization}. Both references provide criteria for
the exponential mean square stabilization of an analogue of
\eqref{IntroPESyst} in discrete time, with an additional non-linear term in
\cite{Diwadkar2015Control}, the signal $\alpha$ being a sequence of
independent identically distributed Bernoulli random variables in $\{0,1\}$ in
\cite{Diwadkar2015Control} and a sequence of real-valued square-integrable
random variables with the same expected value and variance in
\cite{Diwadkar2014Stabilization}. With respect to the setting of the present
paper, apart from the fact that we restrict our attention to more general
systems under the form \eqref{IntroMainSyst} and in continuous time, a major
difference is that we are interested here not only in stabilizability, but
also in obtaining arbitrarily large almost sure exponential decay rates.

In this paper, in order to study the stabilizability by linear feedback laws
of \eqref{IntroMainSyst}, we rewrite it as
\begin{equation}
\label{MainSystAugm}\dot{x}(t)=\widehat{A}x(t)+\widehat{B}_{\alpha
(t)}u_{\alpha(t)}(t),
\end{equation}
where
\begin{equation}
\label{AugmState}x(t) =
\begin{pmatrix}
x_{1}(t)\\
x_{2}(t)\\
\vdots\\
x_{i}(t)\\
\vdots\\
x_{N}(t)
\end{pmatrix}
\in\mathbb{R}^{d},\quad\widehat{A} =
\begin{pmatrix}
A_{1} & 0 & \cdots & 0 & \cdots & 0\\
0 & A_{2} & \cdots & 0 & \cdots & 0\\
\vdots & \vdots & \ddots & \vdots & \ddots & \vdots\\
0 & 0 & \cdots & A_{i} & \cdots & 0\\
\vdots & \vdots & \ddots & \vdots & \ddots & \vdots\\
0 & 0 & \cdots & 0 & \cdots & A_{N}
\end{pmatrix}
,\quad\widehat{B}_{i} =
\begin{pmatrix}
0\\
0\\
\vdots\\
B_{i}\\
\vdots\\
0
\end{pmatrix}
,
\end{equation}
$d=d_{1}+\dotsb+d_{N}$, and $\alpha:\mathbb{R}_{+}\rightarrow\{1,\dotsc,N\}$
is defined from $\alpha_{1},\dotsc,\alpha_{N}:\mathbb{R}_{+}\rightarrow
\{0,1\}$ by setting $\alpha(t)$ to be the unique index $i\in\{1,\dotsc,N\}$
such that $\alpha_{i}(t)=1$. We then look for linear feedback laws of the form
$u_{i}(t)=K_{i}P_{i}x$, where $P_{i}\in\mathcal{M}_{d_{i},d}(\mathbb{R})$ is
the matrix associated with the canonical projection onto the $i$-th factor of
the product $\mathbb{R}^{d}=\mathbb{R}^{d_{1}}\times\dotsb\times
\mathbb{R}^{d_{N}}$. With such feedback laws, \eqref{MainSystAugm} reads
\[
\dot{x}(t)=\left(  \widehat{A}+\widehat{B}_{\alpha(t)}K_{\alpha(t)}
P_{\alpha(t)}\right)  x(t).
\]

Before considering the stabilizability of \eqref{MainSystAugm}, we begin the
paper by the stability analysis of the linear switched system with random
switching
\begin{equation}
\label{MainSystSwitch}\dot{x}(t)=L_{\alpha(t)}x(t),
\end{equation}
where $L_{1},\dotsc,L_{N}\in\mathcal{M}_{d}(\mathbb{R})$ and $\alpha
:\mathbb{R}_{+}\rightarrow\{1,\dotsc,N\}$ is as before. We characterize its
exponential behavior through its Lyapunov exponents, using the classical
Multiplicative Ergodic Theorem due to Oseledets (cf.\ Arnold
\cite{Arnold1998Random}). It turns out that a direct application of this
theorem to systems in continuous time with random switching is not feasible,
since in general they do not define random dynamical systems in the sense of
\cite{Arnold1998Random} (cf.\ Example \ref{Example_noRDS}). Instead, we apply
the Multiplicative Ergodic Theorem to an associated system in discrete time
and then deduce results for the Lyapunov exponents of the continuous-time
system \eqref{MainSystSwitch}. We remark that Lyapunov exponents for
continuous-time systems with random switching are also considered by Li, Chen,
Lam, and Mao in \cite{LCLM12}, but under assumptions on the random switching
signal $\alpha$ guaranteeing that the corresponding switched system is a
random dynamical system, which allows the direct use of the Multiplicative
Ergodic Theorem in continuous time.

The considered linear equations with random switching \eqref{MainSystSwitch}
form Piecewise Deterministic Markov Processes (PDMP). These processes were
introduced in Davis \cite{Davis1984Piecewise} and have since been extensively
studied in the literature. For an analysis of their invariant measures, in
particular, their supports, cf.\ Bakhtin and Hurth \cite{BH12} and
Bena\"{\i}m, Le Borgne, Malrieu, and Zitt \cite{BBMZ15}, also for further
references. An important particular case which also attracts much research
interest is that of Markovian jump linear systems (MJLS), in which one assumes
that the random switching signal is generated by a continuous-time Markov
chain. For more details, we refer to Bolzern, Colaneri, and De Nicolao
\cite{Bolzern2006Almost}, Fang and Loparo \cite{Fang2002Stabilization}, and to
the monograph Costa, Fragoso, and Todorov \cite{Costa2013Continuous}. The case
of nonlinear switched systems with random switching signals has also been
considered in the literature, cf.\ e.g.\ Chatterjee and Liberzon
\cite{Chatterjee2007Stability}, where multiple Lyapunov functions are used to
derive a stability criterion under some slow switching condition that contains
as a particular case switching signals coming from continuous-time Markov
chains. We also remark that several different notions of stability for systems
with random switching have been used in the literature; see, e.g., Feng,
Loparo, Ji, and Chizeck \cite{Feng1992Stochastic} for a comparison between the
usual notions in the context of MJLS. The one considered in this paper is that
of almost sure stability.

The contents of this paper is as follows:

Section \ref{SecMarkov} constructs the random signals $\alpha$ in
\eqref{MainSystAugm} and \eqref{MainSystSwitch}. Example \ref{Example_noRDS}
shows that, in general, \eqref{MainSystSwitch} endowed with such random
switching signals does not define a random dynamical system, and Remark
\ref{Remark_LiChen} discusses the relation to previous works in the
literature. Section \ref{SecDiscretization} introduces an associated system in
discrete time, which defines a random dynamical system in discrete time. We
discuss relations between the Lyapunov exponents for continuous- and
discrete-time systems and state the conclusions we obtain from the
Multiplicative Ergodic Theorem. Section \ref{SecMaxLyap} derives a formula for
the maximal Lyapunov exponent, which is the main ingredient in the stability
analysis of \eqref{MainSystSwitch}. Finally, Section \ref{SecApplications}
presents the main result of this paper, namely that almost sure stabilization
can be achieved for \eqref{IntroMainSyst} with arbitrary decay rate under a
controllability hypothesis.

\medskip

\textbf{Notation:} The sets $\mathbb{N}^{\ast}$ and $\mathbb{N}$ are used to
denote the positive and nonnegative integers, respectively. For $N\in
\mathbb{N}^{\ast}$ we let $\underline{N}:=\{1,...,N\}$ and $\mathbb{R}_{+}:=[0,\infty), \mathbb{R}_{+}^{\ast}:=(0,\infty)$.

\section{Random model for the switching signal}

\label{SecMarkov}

Let $N,d\in\mathbb{N}^{\ast}$ and $L_{1},\dotsc,L_{N}\in\mathcal{M}_{d}(\mathbb{R})$ and consider system \eqref{MainSystSwitch} with a switching
signal $\alpha$ belonging to the set $\mathcal{P}$ defined by
\[
\mathcal{P}:=\left\{  \alpha:\mathbb{R}_{+}\rightarrow\underline{N}\text{
piecewise constant and right continuous}\right\}  .
\]
Recall that a piecewise constant function has only finitely many discontinuity
points on any bounded interval. Given an initial condition $x_{0}\in
\mathbb{R}^{d}$ and $\alpha\in\mathcal{P}$, system \eqref{MainSystSwitch}
admits a unique solution defined on $\mathbb{R}_{+}$, which we denote by
$\varphi_{\mathrm{c}}(\cdot;x_{0},\alpha)$. Furthermore, for $i\in
\underline{N}$, we denote by $\Phi^{i}$ the linear flow defined by the matrix
$L_{i}$, i.e., $\Phi_{t}^{i}=e^{L_{i}t}$ for every $t\in\mathbb{R}$.

In order to describe the random model for the switching signal $\alpha$, let
us first introduce some notation. Given a measurable space $X$, we denote by
$\mathrm{Pr}(X)$ the set of all probability measures on $X$. The set
$\underline{N}$ is assumed to be endowed with the $\sigma$-algebra
$\mathfrak{P}(\underline{N})$ containing all subsets of $\underline{N}$, and
$\mathbb{R}_{+}$ and $\mathbb{R}_{+}^{\ast}$ are assumed to be endowed with
their respective Borel $\sigma$-algebras, denoted for simplicity by
$\mathfrak{B}$ in both cases. Let $\Omega=(\underline{N}\times\mathbb{R}_{+})^{\mathbb{N}^{\ast}}$ and endow $\Omega$ with the standard product
$\sigma$-algebra $\mathfrak{F}=(\mathfrak{P}(\underline{N})\allowbreak
\times\mathfrak{B})^{\mathbb{N}^{\ast}}$ (see, e.g., Halmos \cite[\S 38,
\S 49]{Halmos1974Measure}).

Let $M \in\mathcal{M}_{N}(\mathbb{R})$ be an irreducible right-stochastic
matrix and $p \in\mathbb{R}^{N}$ be its unique invariant probability vector
(regarded here as a row vector). For $i \in\underline N$, let $\mu_{i}
\in\mathrm{Pr}(\mathbb{R}_{+}^{\ast})$ and assume that $\mu_{i}$ has finite
expectation $\tau_{i} = \int_{\mathbb{R}_{+}^{\ast}}t\diff\mu_{i}(t)\in(0,\infty)$ (we also regard $\mu_{i}$ as a Borel probability measure on
$\mathbb{R}_{+}$ whenever necessary). Consider the time-homogeneous
discrete-time Markov process in $\underline{N} \times\mathbb{R}_{+}$ whose
transition probabilities $P: \underline{N}\times\mathbb{R}_{+} \rightarrow
\mathrm{Pr}(\underline{N}\times\mathbb{R}_{+})$ and initial law $\nu_{1}
\in\mathrm{Pr}(\underline N \times\mathbb{R}_{+})$ are given by
\begin{align}
P(i,t)(\{j\}\times U)  &  = M_{ij}\mu_{j}(U), &  &  \forall i,j\in
\underline{N},\;\forall t\in\mathbb{R}_{+},\;\forall U\in\mathfrak{B},\label{DefProbaMarkov}\\
\nu_{1}(\{j\}\times U)  &  = p_{j}\mu_{j}(U), &  &  \forall j\in\underline
{N},\;\forall U\in\mathfrak{B}. \label{DefMarkovNu1}\end{align}
Notice that the associated transition operator $T:\mathrm{Pr}(\underline
{N}\times\mathbb{R}_{+})\rightarrow\mathrm{Pr}(\underline{N}\times
\mathbb{R}_{+})$ of this process is given by
\begin{equation}
\label{Transition}T\nu(\{j\}\times U)=\sum_{i=1}^{N}\nu(\{i\}\times
\mathbb{R}_{+})M_{ij}\mu_{j}(U),\qquad\forall j\in\underline{N},\;\forall
U\in\mathfrak{B},
\end{equation}
and it induces a measure $\mathbb{P }\in\mathrm{Pr}(\Omega)$ defined, for
$n\in\mathbb{N}^{\ast}$, $i_{1},\dotsc,i_{n}\in\underline{N}$, and
$U_{1},\dotsc,U_{n}\in\mathfrak{B}$, by
\begin{equation}
\label{DefMathbbP}\begin{aligned} & \mathbb{P}\left( \left( \{i_{1}\}\times U_{1}\right) \times\left( \{i_{2}\}\times U_{2}\right) \times\dotsb\times\left( \{i_{n}\}\times U_{n}\right) \times\left( \underline{N}\times\mathbb{R}_{+}\right) ^{\mathbb{N}^{\ast}\setminus\underline{n}}\right) \\ {}={} & p_{i_{1}}\mu_{i_{1}}(U_{1})M_{i_{1}i_{2}}\mu_{i_{2}}(U_{2})\dotsm M_{i_{n-1}i_{n}}\mu_{i_{n}}(U_{n}). \end{aligned}
\end{equation}
(For the definition of a discrete-time Markov process in an uncountable set
and its transition probability, initial law, and transition operator, we refer
to Hairer \cite{Hairer2006Ergodic} and Meyn and Tweedie \cite[Chapter
3]{Meyn2009Markov}.)

To construct a random switching signal $\alpha$ from a certain $\omega=
(i_{n}, t_{n})_{n=1}^{\infty}\in\Omega$, we regard $(i_{n})_{n=1}^{\infty}$ as
the sequence of states taken by $\alpha$ and $t_{n}$ as the time spent in the
state $i_{n}$, according to the following definition.

\begin{definition}
\label{DefPmbAlpha} We define the map $\pmb\alpha:\Omega\rightarrow
\mathcal{P}$ as follows: for $\omega=(i_{n},t_{n})_{n=1}^{\infty}\in\Omega$,
we set $s_{0}=0$, $s_{n}=\sum_{k=1}^{n}t_{k}$ for $n\in\mathbb{N}^{\ast}$, and
$\pmb\alpha(\omega)(t)=i_{n}$ for $t\in\left[  s_{n-1}, s_{n}\right)  $,
$n\in\mathbb{N}^{\ast}$.
\end{definition}

This construction of $\pmb\alpha$ amounts to choosing a random initial state
according to the probability law defined by $p$ and, at every switching event
to a state $i$, choosing a random time to stay in this state according to the
law $\mu_{i}$ and a random next state according to the probability law
corresponding to the $i$-th row $(M_{ij})_{j=1}^{N}$ of the matrix $M$. Notice
that $\pmb\alpha(\omega)$ is well-defined only when $\omega$ belongs to the
subset $\Omega_{0} \subset\Omega$ defined by
\begin{equation}
\label{DefiOmega0}\Omega_{0} = \left\{  (i_{n}, t_{n})_{n=1}^{\infty}
\in\Omega\;\middle|\; \sum_{n=1}^{\infty} t_{n} = \infty\right\}  .
\end{equation}
One can easily prove by standard techniques that $\mathbb{P}(\Omega_{0}) = 1$,
yielding that $\pmb\alpha(\omega)$ is well-defined for almost every $\omega
\in\Omega$.

\begin{remark}
In general, since $\mu_{1}, \dotsc, \mu_{N}$ are arbitrary Borel probability
measures on $\mathbb{R}_{+}^{\ast}$, $\pmb\alpha(\omega)$ is not a
continuous-time Markov process on $\underline N$. On the other hand, every
time-homogeneous continuous-time Markov process on $\underline N$ can be
written using the previous definitions by a suitable choice of $M$, $p$, and
$\mu_{1}, \dotsc, \mu_{N}$. Our more general framework covers some important
practical cases that cannot be modeled by continuous-time Markov processes.
For instance, one can model a deterministic switching signal which switches
periodically between $\underline N$ subsystems with prescribed times $T_{1},
\dotsc, T_{N}$ spent in each subsystem by choosing $M$ as an appropriate
irreducible permutation matrix encoding the switching sequence and $\mu_{i} =
\delta_{T_{i}}$ for $i \in\underline N$, where $\delta_{T}$ denotes the Dirac
measure at $T$. In practical implementations, the time spent at a state $i$
may not be exactly equal to $T_{i}$ and some random switches may occur, which
can be modeled in our framework by perturbing the matrix $M$ and choosing
$\mu_{i}$, e.g., as an absolutely continuous measure concentrated around
$T_{i}$.
\end{remark}

In order to consider solutions of \eqref{MainSystSwitch} for signals $\alpha$
chosen randomly according to the previous construction, we use the solution
map $\varphi_{\mathrm{c}}$ of \eqref{MainSystSwitch} to provide the following definition.

\begin{definition}
We define the continuous-time map
\begin{equation}
\label{MainSystContRand}\varphi_{\mathrm{rc}} : \left\{
\begin{aligned} \mathbb{R}_{+} \times\mathbb{R}^{d} \times\Omega_{0} & \to \mathbb{R}^{d}\\ (t; x_{0}, \omega) & \mapsto \varphi_{\mathrm{c}}(t; x_{0}, \pmb\alpha (\omega)). \end{aligned} \right.
\end{equation}
For $x_{0} \in\mathbb{R}^{d} \setminus\{0\}$ and almost every $\omega\in
\Omega$, we define the \emph{Lyapunov exponent} of the continuous-time system
\eqref{MainSystContRand} by
\begin{equation}
\label{ContLyap}\lambda_{\mathrm{rc}}(x_{0}, \omega) = \limsup_{t \to\infty}
\frac{1}{t} \log\left\|  \varphi_{\mathrm{rc}}(t; x_{0}, \omega)\right\|  .
\end{equation}

\end{definition}

The Lyapunov exponent $\lambda_{\mathrm{rc}}$ is used to characterize the
asymptotic behavior of \eqref{MainSystContRand}. A natural idea to obtain
information on such Lyapunov exponents would be to apply the continuous-time
Multiplicative Ergodic Theorem (see, e.g., Arnold \cite[Theorem 3.4.1]{Arnold1998Random}). To do so, $\varphi_{\mathrm{rc}}$ should define a random
dynamical system on $\mathbb{R}^{d}\times\Omega$, i.e., one would have to
provide a metric dynamical system $\theta$ on $\Omega$ --- a measurable
dynamical system $\theta:\mathbb{R}_{+}\times\Omega\rightarrow\Omega$ on
$(\Omega,\mathfrak{F},\mathbb{P})$ such that $\theta_{t}$ preserves
$\mathbb{P}$ for every $t\geq0$ --- in such a way that $\varphi_{\mathrm{rc}}$
becomes a cocycle over $\theta$. However, in general the natural choice for
$\theta$ to obtain the cocycle property for $\varphi_{\mathrm{rc}}$, namely
the time shift, does not define such a measure preserving map, as shown in the
following example.

\begin{example}
\label{Example_noRDS}For $t\geq0$, let $\theta_{t}:\Omega\rightarrow\Omega$ be
defined for almost every $\omega\in\Omega$ as follows. For $\omega
=(i_{j},t_{j})_{j=1}^{\infty}\in\Omega_{0}$, set $s_{0}=0$, $s_{k}=\sum
_{j=1}^{k}t_{j}$ for $k\in\mathbb{N}^{\ast}$. Let $n\in\mathbb{N}^{\ast}$ be
the unique integer such that $t\in\left[  s_{n-1},s_{n}\right)  $. We define
$\theta_{t}(\omega)=(i_{j}^{\ast},t_{j}^{\ast})_{j=1}^{\infty}$ by
$i_{j}^{\ast}=i_{n+j-1}$ for $j\in\mathbb{N}^{\ast}$, $t_{1}^{\ast}=s_{n}-t$,
$t_{j}^{\ast}=t_{n+j-1}$ for $j\geq2$. One immediately verifies that
$\theta_{t}$ corresponds to the time shift in $\mathcal{P}$, i.e., for every
$t,s\geq0$ and $\omega\in\Omega_{0}$, one has
\[
\pmb\alpha(\theta_{t}\omega)(s)=\pmb\alpha(\omega)(t+s).
\]
However, the map $\theta_{t}$ in $(\Omega,\mathfrak{F})$ does not preserve the
measure $\mathbb{P}$ in general. Indeed, suppose that $\mu_{i}=\delta_{1}$ for
every $i\in\underline{N}$, where $\delta_{1}$ denotes the Dirac measure at
$1$. In particular, a set $E\in\mathfrak{F}$ has nonzero measure only if $E$
contains a point $(i_{j},t_{j})_{j=1}^{\infty}$ with $t_{j}=1$ for every
$j\in\mathbb{N}^{\ast}$. For $r\in\mathbb{N}^{\ast}$ and $i_{1},\dotsc
,i_{r}\in\underline{N}$, let
\[
E=\left(  \{i_{1}\}\times\{1\}\right)  \times\dotsb\times\left(
\{i_{r}\}\times\{1\}\right)  \times\left(  \underline{N}\times\mathbb{R}_{+}\right)  ^{\mathbb{N}^{\ast}\setminus\underline{r}}.
\]
Then $\mathbb{P}(E)=p_{i_{1}}M_{i_{1}i_{2}}\dotsm M_{i_{r-1}i_{r}}$, and, for
$t\geq0$, $\theta_{t}^{-1}(E)$ is the set of points $(i_{j}^{\ast},t_{j}^{\ast})_{j=1}^{\infty}$ such that, setting $s_{0}^{\ast}=0$, $s_{k}^{\ast
}=\sum_{j=1}^{k}t_{j}^{\ast}$ for $k\in\mathbb{N}^{\ast}$, and $n\in
\mathbb{N}^{\ast}$ the unique integer such that $t\in\left[  s_{n-1}^{\ast
},s_{n}^{\ast}\right)  $, one has $s_{n}^{\ast}-t=1$, $t_{n+j-1}^{\ast}=1$ for
$j=2,\dotsc,r$, and $i_{n+j-1}^{\ast}=i_{j}$ for $j\in\underline{r}$. If
$t\notin\mathbb{N}$, then $s_{n}^{\ast}=t+1\notin\mathbb{N}$, and thus there
exists $j\in\underline{n}$ such that $t_{j}^{\ast}\not =1$. We have shown
that, if $t\notin\mathbb{N}$, then, for every $\omega=(i_{j}^{\ast},t_{j}^{\ast})_{j=1}^{\infty}\in\theta_{t}^{-1}(E)$, there exists
$j\in\mathbb{N}^{\ast}$ such that $t_{j}^{\ast}\not =1$, and thus
$\mathbb{P}(\theta_{t}^{-1}(E))=0$, hence $\theta_{t}$ does not preserve the
measure $\mathbb{P}$.
\end{example}

\begin{remark}
\label{Remark_LiChen} For some particular choices of $\mu_{1}, \dotsc, \mu
_{N}$, the time-shift $\theta_{t}$ may preserve $\mathbb{P}$, in which case
the continuous-time Multiplicative Ergodic Theorem can be applied directly to
\eqref{MainSystContRand}. This special case falls in the framework of Li,
Chen, Lam, and Mao \cite{LCLM12}. An important particular case where
$\theta_{t}$ preserves $\mathbb{P}$ is when $\mu_{1}, \dotsc, \mu_{N}$ are
chosen in such a way that $\pmb\alpha$ becomes a homogeneous continuous-time
Markov chain, which is the case treated, e.g., in Bolzern, Colaneri, and De
Nicolao \cite{Bolzern2006Almost}, and in Fang and Loparo
\cite{Fang2002Stabilization}. The results we provide in Section
\ref{SecMaxLyap} generalize the corresponding almost sure stability criteria
from \cite{Bolzern2006Almost, Fang2002Stabilization, LCLM12} to randomly
switching signals constructed according to Definition \ref{DefPmbAlpha}.
\end{remark}

\section{Associated discrete-time system and Lyapunov exponents}

\label{SecDiscretization}

Example \ref{Example_noRDS} shows that in general one cannot expect to obtain
a random dynamical system from $\varphi_{\mathrm{rc}}$ in order to apply the
continuous-time Multiplicative Ergodic Theorem. Our strategy to study the
exponential behavior of $\varphi_{\mathrm{rc}}$ relies instead on defining a
suitable discrete-time map $\varphi_{\mathrm{rd}}$ associated with
$\varphi_{\mathrm{rc}}$, in such a way that $\varphi_{\mathrm{rd}}$ does
define a discrete-time random dynamical system --- to which the discrete-time
Multiplicative Ergodic Theorem can be applied --- and that the exponential
behavior of $\varphi_{\mathrm{rc}}$ and $\varphi_{\mathrm{rd}}$ can be compared.

\begin{definition}
\label{DefiVarphiRd} For $\omega= (i_{n}, t_{n})_{n=1}^{\infty}\in\Omega$, we
set $s_{n}(\omega) = \sum_{k=1}^{n} t_{k}$ for $n \in\mathbb{N}^{\ast}$ and
$s_{0}(\omega) = 0$. We define the discrete-time map $\varphi_{\mathrm{rd}}$
by
\begin{equation}
\label{MainSystDiscrRand}\varphi_{\mathrm{rd}} : \left\{
\begin{aligned} \mathbb{N }\times\mathbb{R}^{d} \times\Omega_{0} & \to \mathbb{R}^{d}\\ (n; x_{0}, \omega) & \mapsto \varphi_{\mathrm{rc}}(s_{n}(\omega); x_{0}, \omega). \end{aligned} \right.
\end{equation}
For $x_{0} \in\mathbb{R}^{d} \setminus\{0\}$ and almost every $\omega\in
\Omega$, we define the \emph{Lyapunov exponent} of the discrete-time system
\eqref{MainSystDiscrRand} by
\begin{equation}
\label{DiscrLyap}\lambda_{\mathrm{rd}}(x_{0}, \omega) = \limsup_{n \to\infty}
\frac{1}{n} \log\left\|  \varphi_{\mathrm{rd}}(n; x_{0}, \omega)\right\|  .
\end{equation}

\end{definition}

The map $\varphi_{\mathrm{rd}}$ corresponds to regarding the continuous-time
map $\varphi_{\mathrm{rc}}$ only at the switching times $s_{n}(\omega)$. It is
the solution map of the random discrete-time equation
\begin{equation}
\label{MainSystDiscr}x_{n} = e^{L_{i_{n}} t_{n}} x_{n-1}.
\end{equation}
System \eqref{MainSystDiscr} is obtained from \eqref{MainSystSwitch} by taking
the values of a continuous-time solution at the discrete times $s_{n}(\omega
)$. The sequence $(s_{n}(\omega))_{n=0}^{\infty}$ contains all the
discontinuities of $\pmb\alpha(\omega)$ and may also contain times with
trivial jumps. The Lyapunov exponent $\lambda_{\mathrm{rd}}$ characterizes the
asymptotic behavior of $\varphi_{\mathrm{rd}}$.

Notice that the solution maps $\varphi_{\mathrm{rc}}$ and $\varphi
_{\mathrm{rd}}$ satisfy, for every $x_{0} \in\mathbb{R}^{d}$ and almost every
$\omega= (i_{n}, t_{n})_{n=1}^{\infty}\in\Omega$,
\begin{align}
\varphi_{\mathrm{rc}}(0;x_{0},\omega)  &  =x_{0},\nonumber\\
\varphi_{\mathrm{rc}}(t;x_{0},\omega)  &  = \Phi^{\pmb\alpha(\omega
)(s_{n}(\omega))}_{t - s_{n}(\omega)} \varphi_{\mathrm{rc}}(s_{n}(\omega);x_{0},\omega),\text{ for }n\in\mathbb{N}\text{ and } t\in\left(
s_{n}(\omega),s_{n+1}(\omega)\right]  , \label{ContinuousFlow}\end{align}
and
\begin{align}
\varphi_{\mathrm{rd}}(0;x_{0},\omega)  &  =x_{0},\nonumber\\
\varphi_{\mathrm{rd}}(n+1;x_{0},\omega)  &  = \Phi^{\pmb\alpha(\omega
)(s_{n}(\omega))}_{t_{n+1}}\varphi_{\mathrm{rd}}(n; x_{0},\omega),
\qquad\text{ for }n\in\mathbb{N}. \label{DiscreteFlow}\end{align}

We now prove that $\varphi_{\mathrm{rd}}$ defines a discrete-time random
dynamical system on $\mathbb{R}^{d} \times\Omega$. To do so, we must first
provide a discrete-time metric dynamical system $\theta$ on $(\Omega,
\mathfrak{F}, \mathbb{P})$, which can be chosen simply as the usual shift
operator. Let $\theta: \Omega\to\Omega$ be defined by
\begin{equation}
\label{DefiTheta}\theta((i_{n},t_{n})_{n=1}^{\infty}) = (i_{n+1},t_{n+1})_{n=1}^{\infty}.
\end{equation}
One can easily verify, using \eqref{DefMathbbP} and the fact
that $p M = p$, that the measure $\mathbb{P}$ is invariant under $\theta$, and
thus $\theta$ is a discrete-time metric dynamical system in $(\Omega,
\mathfrak{F}, \allowbreak\mathbb{P})$. Moreover, since $\theta(\Omega_{0}) =
\Omega_{0}$, $\theta$ also defines a metric dynamical system in $(\Omega_{0},
\mathfrak{F}, \mathbb{P})$ (where $\mathfrak{F}$ and $\mathbb{P}$ are
understood to be restricted to $\Omega_{0}$).

Notice that $\theta$ is ergodic in $(\Omega, \mathfrak{F},
\mathbb{P})$. Indeed, given $\nu\in\mathrm{Pr}(\underline N \times
\mathbb{R}_{+})$, let $\mathbb{P}_{\nu}$ be the probability measure on
$\Omega$ associated with the discrete-time Markov process in $\underline N
\times\mathbb{R}_{+}$ with transition probabilities $P$ given by
\eqref{DefProbaMarkov} and initial law $\nu$. One can easily check that
$\mathbb{P}_{\nu}$ is invariant under $\theta$ if and only if $\nu$ coincides
with the initial law $\nu_{1}$ defined in \eqref{DefMarkovNu1}, and thus it
follows from classical ergodicity results for Markov chains (see, e.g., Hairer
\cite[Theorem 5.7]{Hairer2006Ergodic}) that $\theta$ is ergodic in $(\Omega,
\mathfrak{F}, \mathbb{P})$.

Now that we have defined the random discrete-time system
\eqref{MainSystDiscrRand} and provided the metric dynamical system $\theta$,
we can show that the pair $(\theta, \varphi_{\mathrm{rd}})$ defines a random
dynamical system.

\begin{proposition}
$(\theta, \varphi_{\mathrm{rd}})$ is a discrete-time random dynamical system
over $(\Omega, \mathfrak{F},\allowbreak\mathbb{P})$.
\end{proposition}

\begin{proof}
Since $\theta$ is a discrete-time metric dynamical system over $(\Omega,
\mathfrak{F}, \mathbb{P})$, one is only left to show that $\varphi
_{\mathrm{rd}}$ satisfies the cocycle property
\begin{equation}
\label{VarphiCocycle}\varphi_{\mathrm{rd}}(n+m;x_{0},\omega)=\varphi
_{\mathrm{rd}}(n;\varphi_{\mathrm{rd}}(m;x_{0},\omega),\theta^{m}(\omega)),\quad\forall n,m\in\mathbb{N},\;\forall x_{0}\in\mathbb{R}^{d},\;\forall\omega\in\Omega_{0}.
\end{equation}
Let $\omega=(i_{n},t_{n})_{n=1}^{\infty}\in\Omega_{0}$. Then it follows
immediately from the definitions of $\pmb\alpha$ and $s_{n}$ that, for
$n,m\in\mathbb{N}$,
\begin{align*}
s_{n}(\theta^{m}(\omega))  &  =\sum_{k=1}^{n}t_{k+m}=\sum_{k=m+1}^{m+n}
t_{k}=s_{n+m}(\omega)-s_{m}(\omega),\displaybreak[0]\\
\pmb\alpha(\theta^{m}(\omega))(s_{n}(\theta^{m}(\omega)))  &  =i_{n+m} =
\pmb\alpha(\omega)(s_{n+m}(\omega)).
\end{align*}
We prove \eqref{VarphiCocycle} by induction on $n$. When $n=0$,
\eqref{VarphiCocycle} is clearly satisfied for every $m\in\mathbb{N}$,
$x_{0}\in\mathbb{R}^{d}$, and $\omega\in\Omega_{0}$. Suppose now that
$n\in\mathbb{N}$ is such that \eqref{VarphiCocycle} is satisfied for every
$m\in\mathbb{N}$, $x_{0}\in\mathbb{R}^{d}$, and $\omega\in\Omega_{0}$. Using
\eqref{DiscreteFlow}, we obtain
\begin{align*}
&  \varphi_{\mathrm{rd}}(n+1;\varphi_{\mathrm{rd}}(m;x_{0},\omega),\theta
^{m}(\omega))\displaybreak[0]\\
{}={}  &  \Phi_{s_{n+1}(\theta^{m}(\omega))-s_{n}(\theta^{m}(\omega
))}^{\pmb\alpha(\theta^{m}(\omega))(s_{n}(\theta^{m}(\omega)))} \varphi
_{\mathrm{rd}}(n;\varphi_{\mathrm{rd}}(m;x_{0},\omega),\theta^{m}(\omega))
\displaybreak[0]\\
{}={}  &  \Phi_{s_{n+m+1}(\omega)-s_{n+m}(\omega)}^{\pmb\alpha(\omega
)(s_{n+m}(\omega))} \varphi_{\mathrm{rd}}(n+m;x_{0},\omega) =\varphi
_{\mathrm{rd}}(n+m+1;x_{0},\omega),
\end{align*}
which concludes the proof of \eqref{VarphiCocycle}.
\end{proof}

We now compare the asymptotic behavior of \eqref{MainSystContRand} and
\eqref{MainSystDiscrRand} by considering the relation between the Lyapunov
exponents $\lambda_{\mathrm{rc}}(x_{0}, \omega)$ and $\lambda_{\mathrm{rd}}(x_{0}, \omega)$ of the continuous- and discrete-time systems. The following result can be easily obtained from the ergodicity of $\theta$ and Birkhoff's Ergodic Theorem.

\begin{proposition}
\label{PropRegularAlpha} For almost every $\omega\in\Omega$, one has
\begin{equation}
\label{LimitSnOverN}\lim_{n\rightarrow\infty}\frac{s_{n}(\omega)}{n}=\sum_{i=1}^{N}p_{i} \int_{\mathbb{R}_{+}}t\diff\mu_{i}(t)=\sum_{i=1}^{N}p_{i}\tau_{i}=:m.
\end{equation}

\end{proposition}

The next result provides the relation between $\lambda_{\mathrm{rc}}$ and
$\lambda_{\mathrm{rd}}$ in terms of the quantity $m$ defined in \eqref{LimitSnOverN}.

\begin{proposition}
\label{PropLambdaCD} For every $x_{0} \in\mathbb{R}^{d} \setminus\{0\}$ and
almost every $\omega\in\Omega$, the Lyapunov exponents of the continuous- and
discrete-time systems \eqref{MainSystContRand} and \eqref{MainSystDiscrRand},
given by \eqref{ContLyap} and \eqref{DiscrLyap}, are related by
\[
\lambda_{\mathrm{rd}}(x_{0}, \omega) = m \lambda_{\mathrm{rc}}(x_{0},
\omega).
\]

\end{proposition}

\begin{proof}
Let us first show that $\lambda_{\mathrm{rd}}(x_{0}, \omega) \leq m
\lambda_{\mathrm{rc}}(x_{0}, \omega)$. For every $n \in\mathbb{N}^{\ast}$, one
has
\[
\frac{1}{n}\log\left\|  \varphi_{\mathrm{rd}}(n; x_{0}, \omega)\right\|  =
\frac{s_{n}(\omega)}{n}\frac{1}{s_{n}(\omega)} \log\left\|  \varphi
_{\mathrm{rc}}(s_{n}(\omega); x_{0}, \omega)\right\|  .
\]
Moreover
\[
\limsup_{n \to\infty} \frac{1}{s_{n}(\omega)} \log\left\|  \varphi
_{\mathrm{rc}}(s_{n}(\omega); x_{0}, \omega)\right\|  \leq\limsup_{t \to
\infty} \frac{1}{t} \log\left\|  \varphi_{\mathrm{rc}}(t; x_{0},
\omega)\right\|  ,
\]
and then the conclusion follows since $\frac{s_{n}(\omega)}{n} \to m$ as $n
\to\infty$ for almost every $\omega\in\Omega$.

We now turn to the proof of the inequality $\lambda_{\mathrm{rd}}(x_{0},
\omega) \geq m \lambda_{\mathrm{rc}}(x_{0}, \omega)$. Let $C, \gamma> 0$ be
such that $\left\|  \Phi_{t}^{i} x\right\|  \leq C e^{\gamma t} \left\|
x\right\|  $ for every $i\in\underline{N}$, $x\in\mathbb{R}^{d}$, and $t
\geq0$. For $x_{0} \in\mathbb{R}^{d} \setminus\{0\}$ and $t > 0$, let
$n_{t}\in\mathbb{N}$ be the unique integer such that $t \in\left(s_{n_{t}}(\omega),s_{{n_{t}}+1}(\omega)\right]  $, which is well-defined for almost
every $\omega\in\Omega$. Then
\begin{align}
\frac{1}{t}\log\left\|  \varphi_{\mathrm{rc}}(t; x_{0}, \omega)\right\|   &  =
\frac{1}{t} \log\left\|  \Phi_{t-s_{n_{t}}(\omega)}^{\pmb \alpha
(\omega)(s_{n_{t}}(\omega))}\varphi_{\mathrm{rc}}(s_{n_{t}}(\omega); x_{0},
\omega)\right\|  \displaybreak[0]\nonumber\\
&  = \frac{1}{t} \log\left\|  \Phi_{t-s_{n_{t}}(\omega)}^{\pmb\alpha
(\omega)(s_{n_{t}}(\omega))} \varphi_{\mathrm{rd}}(n_{t}; x_{0},
\omega)\right\|  \displaybreak[0]\nonumber\\
&  \leq\frac{\log C}{t} + \gamma\frac{t-s_{n_{t}}(\omega)}{t} + \frac{1}{t}
\log\left\|  \varphi_{\mathrm{rd}}(n_{t}; x_{0}, \omega)\right\|  .
\label{EstimLambdaC}\end{align}
Since $t \in\left(  s_{n_{t}}(\omega),s_{n_{t}+1}(\omega)\right]  $, one has,
for almost every $\omega\in\Omega$,
\begin{equation}
\label{BoundTime}0 \leq\frac{t - s_{n_{t}}(\omega)}{t} \leq\frac{s_{n_{t}+1}(\omega)}{s_{n_{t}}(\omega)} - 1 \xrightarrow[t \to \infty]{} 0,
\end{equation}
where we use \eqref{LimitSnOverN} to obtain that $\frac{s_{n_{t}+1}(\omega
)}{s_{n_{t}}(\omega)} \to1$ as $t \to\infty$. We write $\frac{1}{t} =
\frac{n_{t}}{t} \frac{1}{n_{t}}$. Since $t \in\left(  s_{n_{t}}(\omega),
s_{n_{t}+1}(\omega)\right]  $, one has $\frac{n_{t}}{t} \in\left[  \frac
{n_{t}}{s_{n_{t}+1}(\omega)}, \frac{n_{t}}{s_{n_{t}}(\omega)}\right)  $. Now
\[
\lim_{t \to\infty} \frac{n_{t}}{s_{n_{t}}(\omega)} = \frac{1}{m} \quad\text{
and } \quad\lim_{t \to\infty} \frac{n_{t}}{s_{n_{t}+1}(\omega)} = \lim_{t
\to\infty} \left(  \frac{n_{t}+1}{s_{n_{t}+1}(\omega)} - \frac{1}{s_{n_{t}+1}(\omega)}\right)  =\frac{1}{m},
\]
and thus $\frac{n_{t}}{t} \to\frac{1}{m}$ as $t \to\infty$. Using this fact
and inserting \eqref{BoundTime} into \eqref{EstimLambdaC}, one obtains the
conclusion of the theorem by letting $t \to\infty$.
\end{proof}

The next result relies on the ergodicity of $\theta$ to provide an
evaluation of the average time spent by $\pmb\alpha(\omega)$ in a given state
$i$, generalizing the corresponding classical ergodic theorem for
continuous-time Markov chains to our random model for $\pmb\alpha(\omega)$
(see, e.g., Norris \cite[Theorem 3.8.1]{Norris1998Markov}).

\begin{proposition}
\label{PropProportionTime} Let $i \in\underline N$. For almost every
$\omega\in\Omega$, one has
\[
\lim_{T \to\infty} \frac{\mathcal{L}\{t \in[0, T] \;|\; \pmb\alpha(\omega)(t)
= i\}}{T} = \frac{p_{i} \tau_{i}}{m},
\]
where $\mathcal{L}$ denotes the Lebesgue measure in $\mathbb{R}$.
\end{proposition}

\begin{proof}
Fix $i \in\underline N$. Let $\varphi_{i}: \Omega\to\mathbb{R}_{+}$ be given
by
\[
\varphi_{i}((i_{n}, t_{n})_{n=1}^{\infty}) = \begin{dcases*}
t_{1}, & if $i_{1} = i$, \\
0, & otherwise.
\end{dcases*}
\]
Then, by Birkhoff's Ergodic Theorem, one has, for almost every $\omega
\in\Omega$,
\begin{equation}
\label{LimitBirkhoffVarphiK}\lim_{n \to\infty} \frac{1}{n} \sum_{j=0}^{n-1}
\varphi_{i}(\theta^{j} \omega) = \int_{\Omega} \varphi_{i}(\omega)
\diff \mathbb{P}(\omega) = p_{i} \tau_{i}.
\end{equation}
On the other hand, by definition of $\pmb\alpha$, for almost every $\omega=
(i_{n}, t_{n})_{n=1}^{\infty}\in\Omega$,
\[
\sum_{j=0}^{n-1} \varphi_{i}(\theta^{j} \omega) = \sum_{\substack{j=1 \\i_{j}
= i}}^{n} t_{j} = \mathcal{L}\{t \in[0, s_{n}(\omega)] \;|\; \pmb\alpha
(\omega)(t) = i\}.
\]
Hence it follows from Proposition \ref{PropRegularAlpha} and
\eqref{LimitBirkhoffVarphiK} that, for almost every $\omega\in\Omega$,
\begin{equation}
\label{AverageTimeHolds}\lim_{n \to\infty} \frac{\mathcal{L}\{t \in[0,
s_{n}(\omega)] \;|\; \pmb\alpha(\omega)(t) = i\}}{s_{n}(\omega)} = \lim_{n
\to\infty} \frac{n}{s_{n}(\omega)} \frac{1}{n} \sum_{j=0}^{n-1} \varphi
_{i}(\theta^{j} \omega) = \frac{p_{i} \tau_{i}}{m}.
\end{equation}
Let $\omega\in\Omega$ be such that \eqref{AverageTimeHolds} holds and take $T
\in\mathbb{R}_{+}$. Choose $n_{T} \in\mathbb{N}$ such that $s_{n_{T}}(\omega)
\leq T < s_{n_{T} + 1}(\omega)$. Then
\[
\frac{1}{T}\mathcal{L}\{t \in[0, T] \;|\; \pmb\alpha(\omega)(t) = i\}
\leq\frac{1}{s_{n_{T}}(\omega)}\mathcal{L}\{t \in[0, s_{n_{T} + 1}(\omega)]
\;|\; \pmb\alpha(\omega)(t) = i\}
\]
and
\[
\frac{1}{T}\mathcal{L}\{t \in[0, T] \;|\; \pmb\alpha(\omega)(t) = i\}
\geq\frac{1}{s_{n_{T} + 1}(\omega)}\mathcal{L}\{t \in[0, s_{n_{T}}(\omega)]
\;|\; \pmb\alpha(\omega)(t) = i\}.
\]
The conclusion of the proposition then follows since, by Proposition
\ref{PropRegularAlpha}, $\frac{s_{n+1}(\omega)}{s_{n}(\omega)} \to1$ as $n
\to\infty$ for almost every $\omega\in\Omega$.
\end{proof}

\begin{remark}
The choice of $s_{n}$ in Definition \ref{DefiVarphiRd} is not unique, and one
might be interested in other possible choices. The times $s_{n}(\omega)$
correspond to the transitions of the
discrete-time Markov chain in $\underline N \times \mathbb R_+$ from Section \ref{SecMarkov}.
However, if some of the diagonal elements of $M$ are non-zero, then the
discrete part of the Markov chain, i.e., its component in $\underline{N}$, may
switch from a certain state to itself. In practical situations, it may be
possible to observe only switches between different states, and another
possible choice for $s_{n}(\omega)$ that may be of practical interest is to
consider only the times corresponding to such non-trivial switches. Defining
$\theta$ as the shift to the next different state, $\theta$ defines a metric
dynamical system if we suppose that, instead of having $p M = p$, we have $p
\tilde M = p$, where $\tilde M_{ij} = \frac{M_{ij}}{1 - M_{ii}}$ for $i, j
\in\underline N$ with $i \not = j$ and $\tilde M_{ii} = 0$ for $i
\in\underline N$. (Notice that $M_{ii} \not = 1$ for every $i \in\underline N$
since $M$ is irreducible.) The counterparts of the previous results can be
proved in this framework with no extra difficulty.
\end{remark}

\begin{remark}
The fact that systems \eqref{MainSystSwitch} and \eqref{MainSystDiscr} are
linear has been used only in the proof of Proposition \ref{PropLambdaCD},
where one uses an exponential bound on the growth of the flows $\Phi^{i}_{t} =
e^{L_{i} t}$, namely that there exist constants $C, \gamma> 0$ such that
$\left\|  e^{L_{i} t}\right\|  \leq C e^{\gamma t}$ for every $t \geq0$ and $i
\in\underline N$. If we consider, instead of system \eqref{MainSystSwitch},
the nonlinear switched system
\[
\dot x(t) = f_{\alpha(t)}(x(t)),
\]
where $f_{1}, \dotsc, f_{N}$ are complete vector fields generating flows
$\Phi^{1}, \dotsc, \Phi^{N}$, and modify the dis\-cre\-te-time system
\eqref{MainSystDiscr} accordingly, all the previous results remain true, with
the same proofs, under the additional assumption that there exist constants
$C, \gamma> 0$ such that $\left\|  \Phi^{i}_{t} x\right\|  \leq C e^{\gamma t}
\left\|  x\right\|  $ for every $t \geq0$, $i \in\underline N$, and $x
\in\mathbb{R}^{d}$. However, the next results do not generalize to the
nonlinear framework.
\end{remark}

In order to conclude this section, we apply the discrete-time Multiplicative
Ergodic Theorem (see, e.g., Arnold \cite[Theorem 3.4.1]{Arnold1998Random}) in
the one-sided invertible case to system \eqref{MainSystDiscrRand} and we use
Proposition \ref{PropLambdaCD} to obtain that several of its conclusions also
hold for the continuous-time system \eqref{MainSystContRand}.

Let $L : \Omega\to\mathcal{M}_{d}(\mathbb{R})$ be the function defined for
$\omega= (i_{n},t_{n})_{n=1}^{\infty}$ by $L(\omega) = e^{L_{i_{1}}t_{1}}$, so
that $\varphi_{\mathrm{rd}}(n; x_{0}, \omega) = L(\theta^{n-1} \omega)
\varphi_{\mathrm{rd}}(n-1; x_{0}, \omega)$ for every $x_{0} \in\mathbb{R}^{d}$, $n \in\mathbb{N}^{\ast}$, and almost every $\omega\in\Omega$. For $n
\in\mathbb{N}$ and almost every $\omega\in\Omega$, we denote $\mathit{\Phi}(n,
\omega)$ the linear operator defined by $\mathit{\Phi}(n, \omega) x =
\varphi_{\mathrm{rd}}(n; x, \omega)$ for every $x \in\mathbb{R}^{d}$, which is
thus given by $\mathit{\Phi}(n, \omega) = e^{L_{i_{n}} t_{n}} \dotsm
e^{L_{i_{1}} t_{1}}$ for $\omega= (i_{j}, t_{j})_{j=1}^{\infty} \in\Omega$ and
$n \in\mathbb{N}^{\ast}$.

\begin{proposition}
\label{MET} There exists a measurable subset $\widehat\Omega\subset\Omega$ of
full $\mathbb{P}$-measure and invariant under $\theta$ such that

\begin{enumerate}
\item \label{METPsiExists} for every $\omega\in\widehat\Omega$, the limit
$\Psi(\omega) = \lim_{n \to\infty}\left(  \mathit{\Phi}(n, \omega
)^{{\mathrm{T}}}\mathit{\Phi}(n, \omega)\right)  ^{1/2n}$ exists and is a
positive definite matrix;

\item there exist an integer $q \in\underline d$ and $q$ integers $d_{1} >
\dotsb> d_{q}$ such that, for every $\omega\in\widehat\Omega$, there exist $q$
vector subspaces $V_{1}(\omega), \dotsc, V_{q}(\omega)$ with respective
dimensions $d_{1} > \dotsb> d_{q}$ such that
\[
V_{q}(\omega) \subset\dotsb\subset V_{1}(\omega) = \mathbb{R}^{d},
\]
and $L(\omega) V_{i}(\omega) = V_{i}(\theta(\omega))$ for every $i
\in\underline{q}$;

\item for every $x_{0} \in\mathbb{R}^{d} \setminus\{0\}$ and $\omega
\in\widehat\Omega$, the Lyapunov exponents $\lambda_{\mathrm{rd}}(x_{0},
\omega)$ and $\lambda_{\mathrm{rc}}(x_{0}, \omega)$ exist as limits, i.e.,
\begin{align*}
\lambda_{\mathrm{rd}}(x_{0}, \omega)  &  = \lim_{n \to\infty} \frac{1}{n}
\log\left\|  \varphi_{\mathrm{rd}}(n; x_{0}, \omega)\right\|  ,
\displaybreak[0]\\
\lambda_{\mathrm{rc}}(x_{0}, \omega)  &  = \lim_{t \to\infty} \frac{1}{t}
\log\left\|  \varphi_{\mathrm{rc}}(t; x_{0}, \omega)\right\|  ;
\end{align*}

\item \label{METLyapIff} there exist real numbers $\lambda_{1}^{\mathrm{d}} >
\dotsb> \lambda_{q}^{\mathrm{d}}$ and $\lambda_{1}^{\mathrm{c}} > \dotsb>
\lambda_{q}^{\mathrm{c}}$ such that, for every $i \in\underline{q}$ and
$\omega\in\widehat\Omega$,
\[
\lambda_{\mathrm{rd}}(x_{0}, \omega) = \lambda_{i}^{\mathrm{d}} \iff
\lambda_{\mathrm{rc}}(x_{0}, \omega) = \lambda_{i}^{\mathrm{c}} \iff x_{0} \in
V_{i}(\omega)\setminus V_{i+1}(\omega),
\]
where $V_{q + 1}(\omega) = \{0\}$;

\item \label{METSpecPsi} for every $\omega\in\widehat{\Omega}$, the
eigenvalues of $\Psi(\omega)$ are $e^{\lambda_{1}^{\mathrm{d}}}>\dotsb
>e^{\lambda_{q}^{\mathrm{d}}}$, and their respective algebraic multiplicities are $m_i = d_i - d_{i+1}$, with $d_{q+1} = 0$.
\end{enumerate}
\end{proposition}

\begin{proof}
Let us show that Multiplicative Ergodic Theorem can be applied to the random
dynamical system $(\theta,\varphi_{\mathrm{rd}})$. Recall that there are
$C\geq1$, $\gamma>0$ such that, for every $i\in\underline{N}$ and
$t\in\mathbb{R}$, $\left\Vert e^{L_{i}t}\right\Vert \leq Ce^{\gamma\left\vert
t\right\vert }$. Then, for $\omega=(i_{n},t_{n})_{n=1}^{\infty}\in\Omega_{0}$,
$\log^{+}\left\Vert L(\omega)^{\pm1}\right\Vert \leq\log C+\gamma t_{1}$, so
that
\[
\int_{\Omega}\log^{+}\left\Vert L(\omega)^{\pm1}\right\Vert \diff\mathbb{P}(\omega)\leq\log C+\gamma\sum_{i=1}^{N}p_{i}\tau_{i}<\infty.
\]
Then the Multiplicative Ergodic Theorem can be applied to $(\theta
,\varphi_{\mathrm{r}d})$, yielding all the conclusions for $\Psi$, $q$,
$d_{i}$, $V_{i}$, $\lambda_{\mathrm{rd}}(x_{0},\omega)$, and $\lambda
_{i}^{\mathrm{d}}$. The conclusions concerning $\lambda_{\mathrm{rc}}(x_{0},\omega)$ and $\lambda_{i}^{\mathrm{c}}$ in \ref{METLyapIff} follow from
Proposition \ref{PropLambdaCD}, with $\lambda_{i}^{\mathrm{c}}=\frac{1}{m}\lambda_{i}^{\mathrm{d}}$. One is now left to show that the Lyapunov
exponent $\lambda_{\mathrm{rc}}(x_{0},\omega)$ exists as a limit.

Notice that $\left\Vert e^{-L_{i}t}x\right\Vert \leq Ce^{\gamma t}\left\Vert
x\right\Vert $ for every $i\in\underline{N}$, $x\in\mathbb{R}^{d}$ and
$t\geq0$, and hence $\left\Vert e^{L_{i}t}x\right\Vert \geq C^{-1}e^{-\gamma
t}\left\Vert x\right\Vert $. Let $t>0$ and choose $n_{t}\in\mathbb{N}$ such
that $t\in\left(  s_{n_{t}}(\omega),s_{n_{t}+1}(\omega)\right]  $. Then,
proceeding as in \eqref{EstimLambdaC}, one gets
\[
\frac{1}{t}\log\left\Vert \varphi_{\mathrm{rc}}(t;x_{0},\omega)\right\Vert
\geq-\frac{\log C}{t}-\gamma\frac{t-s_{n_{t}}}{t}+\frac{1}{t}\log\left\Vert
\varphi_{\mathrm{rd}}(n_{t};x_{0},\omega)\right\Vert .
\]
Using \eqref{BoundTime}, we thus obtain that
\[
\liminf_{t\rightarrow\infty}\frac{1}{t}\log\left\Vert \varphi_{\mathrm{rc}}(t;x_{0},\omega)\right\Vert \geq\frac{1}{m}\lambda_{\mathrm{rd}}(x_{0},\omega) = \lambda_{\mathrm{rc}}(x_{0},\omega),
\]
which yields the existence of the limit.
\end{proof}

\section{The maximal Lyapunov exponent}

\label{SecMaxLyap}

We are interested in this section in the maximal Lyapunov exponents for
systems \eqref{MainSystContRand} and \eqref{MainSystDiscrRand}, i.e., the real
numbers $\lambda_{1}^{\mathrm{c}}$ and $\lambda_{1}^{\mathrm{d}}$ from
Proposition \ref{MET}\ref{METLyapIff}. We denote these numbers by
$\lambda_{\max}^{\mathrm{c}}$ and $\lambda_{\max}^{\mathrm{d}}$, respectively.
Before proving the main results of this section, we state the following lemma,
which shows that the Gelfand formula for the spectral radius $\rho$ holds
uniformly over compact sets of matrices. This follows from the estimates
derived in Green \cite[Section 3.3]{Green1996Uniform}. For the reader's
convenience, we provide a proof.

\begin{lemma}
\label{LemmGelfandUnif} Let $\mathcal{A} \subset\mathcal{M}_{d}(\mathbb{R})$
be a compact set of matrices. Then the limit
\[
\lim_{n\rightarrow\infty}\left\Vert A^{n}\right\Vert ^{1/n}=\rho(A)
\]
is uniform over $\mathcal{A}$.
\end{lemma}

\begin{proof}
Let $\varepsilon>0$ and define a continuous function $F:\mathcal{A}\rightarrow\mathcal{M}_{d}(\mathbb{R})$ by
\[
F(A)=\frac{1}{\rho(A)+\varepsilon}A.
\]
Then $F(\mathcal{A})$ is compact and for every $F(A)\in F(\mathcal{A})$ its
spectral radius is $\rho(F(A))=\frac{\rho(A)}{\rho(A)+\varepsilon}<1$. Fix
$A\in\mathcal{A}$. Then (see, e.g., Horn and Johnson \cite[Lemma
5.6.10]{Horn2013Matrix}) there is a norm $\left\Vert \cdot\right\Vert _{A}$ in
$\mathbb{R}^{d}$ with $\left\Vert F(A)\right\Vert _{A}<\frac{1+\rho(F(A))}{2}$. Then for all $B$ in a neighborhood $U$ of $A$
\[
\left\Vert F(B)\right\Vert _{A}<\frac{1+\rho(F(A))}{2}.
\]
Since all norms on $\mathcal{M}_{d}(\mathbb{R})$ are equivalent, there is
$\beta_{A}>0$ such that for all $B\in U$
\[
\left\Vert F(B)^{n}\right\Vert \leq\beta_{A}\left\Vert F(B)^{n}\right\Vert
_{A}\leq\beta_{A}\left\Vert F(B)\right\Vert _{A}^{n}\leq\beta_{A}\left(
\frac{1+\rho(F(A))}{2}\right)  ^{n}.
\]
Then there is $N\in\mathbb{N}^{\ast}$, depending only on $A$ and $\varepsilon
$, such that for all $n\geq N$ and all $B\in U$,
\[
\frac{1}{\rho(B)+\varepsilon}\left\Vert B^{n}\right\Vert ^{1/n}=\left\Vert
F(B)^{n}\right\Vert ^{1/n}<1,
\]
implying $\left\Vert B^{n}\right\Vert ^{1/n}<\rho(B)+\varepsilon$. Since this
holds for every $B$ in a neighborhood $U$ of $A$ and $\left\Vert
B^{n}\right\Vert ^{1/n}\geq\rho(B)$ for every $n\in\mathbb{N}^{\ast}$, one
obtains that the convergence in $U$ is uniform, and the assertion follows by
compactness of $\mathcal{A}$.
\end{proof}

We can now prove our first result regarding the characterization of
$\lambda_{\max}^{\mathrm{c}}$ and $\lambda_{\max}^{\mathrm{d}}$.

\begin{proposition}
\label{PropLambdaMax} For almost every $\omega\in\Omega$, we have
\begin{equation}
\label{LambdaMaxInvariant}\lambda_{\max}^{\mathrm{d}}=\lim_{n\rightarrow
\infty}\frac{1}{n}\log\left\Vert \mathit{\Phi}(n,\omega)\right\Vert .
\end{equation}
Moreover,
\begin{equation}
\label{LambdaMaxErgodic}\lambda_{\max}^{\mathrm{d}}\leq\inf_{n\in
\mathbb{N}^{\ast}}\frac{1}{n}\int_{\Omega}\log\left\Vert \mathit{\Phi
}(n,\omega)\right\Vert \diff\mathbb{P}(\omega)=\lim_{n\rightarrow\infty}\frac{1}{n}\int_{\Omega}\log\left\Vert \mathit{\Phi}(n,\omega)\right\Vert
\diff\mathbb{P}(\omega).
\end{equation}

\end{proposition}

\begin{proof}
Notice that \eqref{LambdaMaxInvariant} and \eqref{LambdaMaxErgodic} do not
depend on the norm in $\mathcal{M}_{d}(\mathbb{R})$. We fix in this proof the
norm induced by the Euclidean norm in $\mathbb{R}^{d}$, given by $\left\|
A\right\|  = \sqrt{\rho(A^{{\mathrm{T}}}A)}$. Notice that, in this case,
$\left\|  A^{{\mathrm{T}}}A\right\|  = \sqrt{\rho((A^{{\mathrm{T}}}A)^{2})} =
\rho(A^{{\mathrm{T}}}A) = \left\|  A\right\|  ^{2}$.

By Proposition \ref{MET}\ref{METSpecPsi}, $e^{\lambda_{\max}^{\mathrm{d}}}$ is
the spectral radius $\rho(\Psi(\omega))$ of $\Psi(\omega)$ for almost every
$\omega\in\Omega$. By continuity of the spectral radius and Proposition
\ref{MET}\ref{METPsiExists}, one then gets that
\begin{align}
e^{\lambda_{\max}^{\mathrm{d}}}  &  =\lim_{n\rightarrow\infty}\rho\left[
\left(  \mathit{\Phi}(n,\omega)^{{\mathrm{T}}}\mathit{\Phi}(n,\omega)\right)
^{1/2n}\right] \nonumber\\
&  = \lim_{n\rightarrow\infty}\lim_{k\rightarrow\infty}\left\Vert \left(
\mathit{\Phi}(n,\omega)^{{\mathrm{T}}}\mathit{\Phi}(n,\omega)\right)
^{k/2n}\right\Vert ^{1/k}, \label{LimNLimK}\end{align}
using also Gelfand's Formula for the spectral radius. The sequence of matrices
$\left(  \left(  \mathit{\Phi}(n,\omega)^{{\mathrm{T}}}\mathit{\Phi}(n,\omega)\right)  ^{1/2n}\right)  _{n=1}^{\infty}$ converges to $\Psi
(\omega)$, hence this sequence is bounded in $\mathcal{M}_{d}(\mathbb{R})$. By
Lemma \ref{LemmGelfandUnif}, the limit in Gelfand's Formula is uniform, which
shows that one can take the limit in \eqref{LimNLimK} along the line $k=2n$ to
obtain
\[
e^{\lambda_{\max}^{\mathrm{d}}}=\lim_{n\rightarrow\infty}\left\Vert
\mathit{\Phi}(n,\omega)^{{\mathrm{T}}}\mathit{\Phi}(n,\omega)\right\Vert
^{1/2n}=\lim_{n\rightarrow\infty}\left\Vert \mathit{\Phi}(n,\omega)\right\Vert
^{1/n}.
\]
Hence \eqref{LambdaMaxInvariant} follows by taking the logarithm.

In order to prove \eqref{LambdaMaxErgodic}, fix $m\in\mathbb{N}^{\ast}$. By
\eqref{LambdaMaxInvariant}, for almost every $\omega\in\Omega$,
\begin{equation}
\label{Lambda1Subsequence}\lambda_{\max}^{\mathrm{d}}=\lim_{n\rightarrow
\infty}\frac{1}{nm}\log\left\Vert \mathit{\Phi}(nm,\omega)\right\Vert .
\end{equation}
One has $\mathit{\Phi}(nm,\omega)=\mathit{\Phi}(m,\theta^{(n-1)m}\omega
)\dotsm\mathit{\Phi}(m,\theta^{m}\omega)\mathit{\Phi}(m,\omega)$, and thus
\begin{equation}
\label{SubadditivityM}\frac{1}{nm}\log\left\Vert \mathit{\Phi}(nm,\omega
)\right\Vert \leq\frac{1}{nm}\sum_{k=0}^{n-1}\log\left\Vert \mathit{\Phi
}(m,\theta^{mk}\omega)\right\Vert .
\end{equation}
Since $\theta^{m}$ preserves $\mathbb{P}$ and $\log\left\Vert \mathit{\Phi
}(m,\cdot)\right\Vert \in L^{1}(\Omega,\mathbb{R})$, Birkhoff's Ergodic
Theorem shows that
\begin{equation}
\label{LimitBirkhoff}\lim_{n\rightarrow\infty}\frac{1}{nm}\sum_{k=0}^{n-1}\log\left\Vert \mathit{\Phi}(m,\theta^{mk}\omega)\right\Vert =\frac{1}{m}\int_{\Omega}\log\left\Vert \mathit{\Phi}(m,\omega)\right\Vert
\diff \mathbb{P}(\omega).
\end{equation}
Combining \eqref{Lambda1Subsequence}, \eqref{SubadditivityM}, and
\eqref{LimitBirkhoff}, one obtains the inequality in \eqref{LambdaMaxErgodic}.
The sequence $\left(  \int_{\Omega}\log\left\Vert \mathit{\Phi}(n,\omega
)\right\Vert \diff \mathbb{P}(\omega)\right)  _{n} $ is subadditive, since
$\mathit{\Phi}(n+m,\omega)=\mathit{\Phi}(m,\theta^{n}\omega)\mathit{\Phi
}(n,\omega)$ for $n,m\in\mathbb{N}$ and $\theta$ preserves $\mathbb{P}$. This
subadditivity implies that the equality in \eqref{LambdaMaxErgodic} holds.
\end{proof}

Under some extra assumptions on the probability measures $\mu_{i}$, $i
\in\underline N$, one obtains that the inequality in \eqref{LambdaMaxErgodic}
is actually an equality.

\begin{proposition}
\label{PropMaxLyapInf} Suppose there exists $r > 1$ such that, for every $i
\in\underline N$, one has $\int_{(0, \infty)} t^{r} \diff \mu_{i}(t) < \infty
$. Then $\lambda_{\max}^{\mathrm{d}}$ is given by
\[
\lambda_{\max}^{\mathrm{d}} = \inf_{n \in\mathbb{N}^{\ast}} \frac{1}{n}
\int_{\Omega} \log\left\|  \mathit{\Phi}(n, \omega)\right\|  \diff \mathbb{P}(\omega) = \lim_{n \to\infty} \frac{1}{n} \int_{\Omega} \log\left\|
\mathit{\Phi}(n, \omega)\right\|  \diff \mathbb{P}(\omega).
\]

\end{proposition}

\begin{proof}
One clearly has, using \eqref{LambdaMaxInvariant}, that
\[
\lambda_{\max}^{\mathrm{d}}=\int_{\Omega}\lambda_{\max}^{\mathrm{d}}
\diff \mathbb{P}(\omega)=\int_{\Omega}\lim_{n\rightarrow\infty}\frac{1}{n}\log\left\Vert \mathit{\Phi}(n,\omega)\right\Vert \diff \mathbb{P}(\omega).
\]
The theorem is proved if we show one can exchange the limit and the integral
in the above expression, which we do by applying Vitali's convergence theorem
(see, e.g., Rudin \cite[Chapter 6]{Rudin1987Real}). We are thus left to show
that the sequence of functions $\left(  \frac{1}{n}\log\left\Vert
\mathit{\Phi}(n,\cdot)\right\Vert \right)  _{n = 1}^{\infty}$ is uniformly
integrable, i.e., for every $\varepsilon>0$, there exists $\delta>0$ such
that, for every $E\in\mathfrak{F}$ with $\mathbb{P}(E)<\delta$, one has
$\frac{1}{n}\left\vert \int_{E}\log\left\Vert \mathit{\Phi}(n,\omega
)\right\Vert \diff \mathbb{P}(\omega)\right\vert <\varepsilon$.

For $\omega= (i_{n}, t_{n})_{n=1}^{\infty} \in\Omega_{0}$ and $n \in
\mathbb{N}^{\ast}$, one has $\mathit{\Phi}(n, \omega) = e^{L_{i_{n}} t_{n}}
\dotsm e^{L_{i_{1}} t_{1}}$ and $\mathit{\Phi}(0, \omega) = \id_{d}$. Let $C,
\gamma> 0$ be such that $\left\|  e^{L_{i} t}\right\|  \leq C e^{\gamma\left|
t\right|  }$ for every $i \in\underline N$ and $t
\in\mathbb{R}$. For every $n \in\mathbb{N}$, since $\mathit{\Phi}(n+1, \omega)
= e^{L_{i_{n+1}} t_{n+1}} \mathit{\Phi}(n, \omega)$ and $\mathit{\Phi}(n,
\omega) = e^{- L_{i_{n+1}} t_{n+1}} \mathit{\Phi}(n+1, \omega)$, one obtains
that
\[
C^{-1} e^{-\gamma t_{n+1}} \left\|  \mathit{\Phi}(n, \omega)\right\|
\leq\left\|  \mathit{\Phi}(n+1, \omega)\right\|  \leq C e^{\gamma t_{n+1}}
\left\|  \mathit{\Phi}(n, \omega)\right\|  ,
\]
and thus an inductive argument yields, for $n \in\mathbb{N}^{\ast}$,
\[
C^{-n} e^{-\gamma s_{n}(\omega)} \leq\left\|  \mathit{\Phi}(n, \omega
)\right\|  \leq C^{n} e^{\gamma s_{n}(\omega)},
\]
where $s_{n}(\omega) = \sum_{i=1}^{n} t_{i}$. Then
\[
\big\lvert\log\left\|  \mathit{\Phi}(n, \omega)\right\|  \big\rvert  \leq n
\log C + \gamma s_{n}(\omega).
\]
Hence, it suffices to show that the sequence $\left(  \frac{s_{n}}{n}\right)
_{n=1}^{\infty}$ is uniformly integrable.

For $n \in\mathbb{N}^{\ast}$ and $E \in\mathfrak{F}$, we have, by H\"older's
inequality,
\begin{align}
\int_{E} \frac{s_{n}(\omega)}{n} \diff \mathbb{P}(\omega)  &  = \frac{1}{n}
\sum_{i=1}^{n} \int_{E} t_{i} \diff \mathbb{P}(\omega)\nonumber\\
&  \leq\frac{1}{n} \sum_{i=1}^{n} \left(  \int_{\Omega} t_{i}^{r}
\diff \mathbb{P}(\omega)\right)  ^{\frac{1}{r}} \mathbb{P}(E)^{\frac
{1}{r^{\prime}}} \leq K^{\frac{1}{r}} \mathbb{P}(E)^{\frac{1}{r^{\prime}}},
\label{UnifIntegrable}\end{align}
where $r^{\prime}\in(1, \infty)$ is such that $\frac{1}{r} + \frac
{1}{r^{\prime}} = 1$ and $K = \max_{i \in\underline N} \int_{(0, \infty)}
t^{r} \diff \mu_{i}(t) < \infty$. Equation \eqref{UnifIntegrable} establishes
the uniform integrability of $\left(  \frac{s_{n}}{n}\right)  _{n = 1}^{\infty}$, which yields the result.
\end{proof}

As an immediate consequence of Proposition \ref{PropRegularAlpha}, Proposition
\ref{PropLambdaCD}, Proposition \ref{PropLambdaMax}, and Proposition
\ref{PropMaxLyapInf}, we obtain the following result.

\begin{corollary}
\label{CoroLambdaMaxCD} The maximal Lyapunov exponents $\lambda_{\max
}^{\mathrm{c}}$ and $\lambda_{\max}^{\mathrm{d}}$ satisfy
\begin{equation}
\label{LambdaMaxDIneq}m \lambda_{\max}^{\mathrm{c}} = \lambda_{\max
}^{\mathrm{d}} \leq\inf_{n\in\mathbb{N}^{\ast}}\frac{1}{n}\int_{\Omega}\log\left\Vert \mathit{\Phi}(n,\omega)\right\Vert \diff \mathbb{P}(\omega).
\end{equation}
In particular, if
\begin{equation}
\label{CondNSExpoStab}\text{there exists } n \in\mathbb{N}^{\ast}\text{ such
that } \int_{\Omega}\log\left\Vert \mathit{\Phi}(n,\omega)\right\Vert
\diff \mathbb{P}(\omega) < 0,
\end{equation}
then systems \eqref{MainSystContRand} and \eqref{MainSystDiscrRand} are almost
surely exponentially stable.

If we have further that there exists $r > 1$ such that $\int_{\mathbb{R}_{+}}
t^{r} \diff \mu_{i}(t) < \infty$ for every $i \in\underline N$, then the
inequality in \eqref{LambdaMaxDIneq} is an equality and \eqref{CondNSExpoStab}
is equivalent to the almost sure exponential stability of
\eqref{MainSystContRand} and to the almost sure exponential stability of \eqref{MainSystDiscrRand}.
\end{corollary}

We conclude this section with the following characterization of a weighted sum
of the Lyapunov exponents $\lambda_{i}^{\mathrm{d}}$, $i\in\underline{N}$.

\begin{proposition}
Suppose there exists $r > 1$ such that, for every $i \in\underline N$, one has
$\int_{(0, +\infty)} t^{r} \diff \mu_{i}(t) < \infty$. Then
\begin{equation}
\label{SumLyapExp}\sum_{i=1}^{q} m_{i} \lambda_{i}^{\mathrm{d}} = \sum
_{i=1}^{N} p_{i} \tau_{i} \Tr(L_{i}),
\end{equation}
where $m_{i}$ is as in Proposition \ref{MET}\ref{METSpecPsi}.
\end{proposition}

\begin{proof}
Thanks to Proposition \ref{MET}\ref{METSpecPsi}, one obtains that, for almost
every $\omega= (i_{n},\allowbreak t_{n})_{n=1}^{\infty}\in\Omega$,
\[
\det\Psi(\omega) = \prod_{i=1}^{q} e^{m_{i} \lambda_{i}^{\mathrm{d}}},
\]
which yields
\begin{align*}
\sum_{i=1}^{q} m_{i} \lambda_{i}^{\mathrm{d}}  &  = \log\det\Psi(\omega) =
\lim_{n \to\infty} \log\det\left(  \mathit{\Phi}(n, \omega)^{{\mathrm{T}}
}\mathit{\Phi}(n, \omega)\right)  ^{1/2n} \displaybreak[0]\\
&  = \lim_{n \to\infty} \log\left(  \prod_{k=1}^{n} \det e^{L_{i_{k}} t_{k}}\right)  ^{1/n} = \lim_{n \to\infty} \frac{1}{n} \sum_{k=1}^{n} t_{k}
\Tr(L_{i_{k}}).
\end{align*}
Then
\begin{align*}
\sum_{i=1}^{q} m_{i} \lambda_{i}^{\mathrm{d}}  &  = \int_{\Omega} \sum
_{i=1}^{q} m_{i} \lambda_{i}^{\mathrm{d}} \diff \mathbb{P}(\omega) =
\int_{\Omega} \lim_{n \to\infty} \frac{1}{n} \sum_{k=1}^{n} t_{k}
\Tr(L_{i_{k}}) \diff \mathbb{P}(\omega) \displaybreak[0]\\
&  = \lim_{n \to\infty} \frac{1}{n} \sum_{k=1}^{n} \int_{\Omega} t_{k}
\Tr(L_{i_{k}}) \diff \mathbb{P}(\omega) = \sum_{i=1}^{N} p_{i} \tau_{i}
\Tr(L_{i}),
\end{align*}
where we exchange limit and integral thanks to Vitali's convergence theorem
and to the fact that $\left(  \frac{s_{n}(\omega)}{n}\right)  _{n=1}^{\infty}=
\left(  \frac{1}{n} \sum_{k=1}^{n} t_{k}\right)  _{n=1}^{\infty}$ is uniformly
integrable, as shown in the proof of Proposition \ref{PropMaxLyapInf}.
\end{proof}

\section{Main result}

\label{SecApplications}

In this section, we use the stability criterion from Corollary
\ref{CoroLambdaMaxCD} to study the stabilization by linear feedback laws of
\eqref{IntroMainSyst}. As stated in the Introduction, we write
\eqref{IntroMainSyst} under the form \eqref{MainSystAugm}, which is a switched
control system with dynamics given by the $N$ equations $\dot x = \widehat A x
+ \widehat B_{i} u_{i}$, $i \in\underline N$.

We consider system \eqref{MainSystAugm} in a probabilistic setting by taking
random signals $\pmb\alpha(\omega)$ as in Definition \ref{DefPmbAlpha}, i.e.,
the random control system $\dot x(t) = \widehat A x(t) + \widehat
B_{\pmb\alpha(\omega)(t)} \allowbreak u_{\pmb\alpha(\omega)(t)}(t)$. The
problem treated in this section is the arbitrary rate stabilizability of this
system by linear feedback laws $u_{i} = K_{i} P_{i} x$, $i \in\underline N$,
where we recall that $P_{i} \in\mathcal{M}_{d_{i}, d}(\mathbb{R})$ is the
projection onto the $i$-th factor of $\mathbb{R}^{d} = \mathbb{R}^{d_{1}}
\times\dotsb\times\mathbb{R}^{d_{N}}$. More precisely, we consider the
closed-loop random switched system
\begin{equation}
\label{ClosedLoop}\dot x(t) = \left(  \widehat A + \widehat B_{\pmb\alpha
(\omega)(t)} K_{\pmb\alpha(\omega)(t)} P_{\pmb\alpha(\omega)(t)}\right)  x(t).
\end{equation}
We wish to know if, given $\lambda\in\mathbb{R}$, there exist matrices $K_{i}
\in\mathcal{M}_{m_{i}, d_{i}}(\mathbb{R})$, $i \in\underline N$, such that the
maximal Lyapunov exponent $\lambda_{\max}^{\mathrm{c}}$ of the continuous-time
system \eqref{ClosedLoop}, defined as in Section \ref{SecMaxLyap}, satisfies
$\lambda_{\max}^{\mathrm{c}} \leq\lambda$. Our main result is the following,
which states that this is true under the controllability of $(A_{i}, B_{i})$
for every $i \in\underline N$, thus implying that arbitrary decay rates are achievable.

\begin{theorem}
\label{MainTheoApp} Let $N \in\mathbb{N}^{\ast}$, $d_{1}, \dotsc, d_{N},
m_{1}, \dotsc, m_{N} \in\mathbb{N}$, $A_{i} \in\mathcal{M}_{d_{i}}(\mathbb{R})$, $B_{i} \in\mathcal{M}_{d_{i}, m_{i}}(\mathbb{R})$ for $i
\in\underline N$, and assume that $(A_{i}, B_{i})$ is controllable for every
$i \in\underline N$. Define $\widehat A$ and $\widehat B_{i}$ as in
\eqref{AugmState}. Then, for every $\lambda\in\mathbb{R}$, there exist
matrices $K_{i} \in\mathcal{M}_{m_{i}, d_{i}}(\mathbb{R})$, $i \in\underline
N$, such that the maximal Lyapunov exponent $\lambda_{\max}^{\mathrm{c}}$ of
the closed-loop random switched system \eqref{ClosedLoop} satisfies
$\lambda_{\max}^{\mathrm{c}} \leq\lambda$.
\end{theorem}

\begin{proof}
Let $C \geq1,\; \beta> 0$ be such that, for every $i \in\underline N$ and
every $t \geq0$, $\left\|  e^{A_{i} t}\right\|  \leq C e^{\beta t}$. Thanks to
Cheng, Guo, Lin, and Wang \cite[Proposition 2.1]{Cheng2004Note} (see also
\cite[Proposition 1]{Cheng2005Erratum}), we may assume that $C$ is chosen
large enough such that the following property holds: there exists $D
\in\mathbb{N}^{\ast}$ such that, for every $\gamma\geq1$ and $i \in\underline
N$, there exists a matrix $K_{i} \in\mathcal{M}_{m_{i}, d_{i}}(\mathbb{R})$
with
\begin{equation}
\label{EstimCGLW}\left\|  e^{(A_{i} + B_{i} K_{i}) t}\right\|  \leq C
\gamma^{D} e^{-\gamma t}, \qquad\forall t \in\mathbb{R}_{+}.
\end{equation}
Let $\widehat K_{i} = K_{i} P_{i} \in\mathcal{M}_{m_{i}, d}(\mathbb{R})$.
Then
\[
\widehat A + \widehat B_{i} \widehat K_{i} =
\begin{pmatrix}
A_{1} & 0 & \cdots & 0 & \cdots & 0\\
0 & A_{2} & \cdots & 0 & \cdots & 0\\
\vdots & \vdots & \ddots & \vdots & \ddots & \vdots\\
0 & 0 & \cdots & A_{i} + B_{i} K_{i} & \cdots & 0\\
\vdots & \vdots & \ddots & \vdots & \ddots & \vdots\\
0 & 0 & \cdots & 0 & \cdots & A_{N}\end{pmatrix}
,
\]
and thus, for every $t \in\mathbb{R}$,
\[
e^{(\widehat A + \widehat B_{i} \widehat K_{i}) t} =
\begin{pmatrix}
e^{A_{1} t} & 0 & \cdots & 0 & \cdots & 0\\
0 & e^{A_{2} t} & \cdots & 0 & \cdots & 0\\
\vdots & \vdots & \ddots & \vdots & \ddots & \vdots\\
0 & 0 & \cdots & e^{(A_{i} + B_{i} K_{i}) t} & \cdots & 0\\
\vdots & \vdots & \ddots & \vdots & \ddots & \vdots\\
0 & 0 & \cdots & 0 & \cdots & e^{A_{N} t}\end{pmatrix}
.
\]

Since $M$ is irreducible and $p$ is invariant under $M$, we have $p_{i} > 0$
for every $i \in\underline N$. The irreducibility of $M$ also provides the
existence of $r \geq N$ and $(i_{1}^{\ast}, \dotsc, i_{r}^{\ast})
\in\underline N^{r}$ such that $\{i_{1}^{\ast}, \dotsc, i_{r}^{\ast}\} =
\underline N$ and $M_{i_{1}^{\ast}i_{2}^{\ast}} \dotsm M_{i_{r-1}^{\ast}
i_{r}^{\ast}} > 0$. In order to apply Corollary \ref{CoroLambdaMaxCD},
consider
\begin{align}
&  \int_{\Omega}\log\left\|  \mathit{\Phi}(r, \omega)\right\|
\diff \mathbb{P}(\omega) = \sum_{(i_{1}, \dotsc, i_{r}) \in{\underline N}^{r}}
p_{i_{1}} M_{i_{1} i_{2}} \dotsm M_{i_{r-1} i_{r}}\nonumber\\
&  \hspace{1.5cm} \cdot\int_{(0, \infty)^{r}} \log\left\|  e^{(\widehat A +
\widehat B_{i_{r}} \widehat K_{i_{r}}) t_{r}} \dotsm e^{(\widehat A + \widehat
B_{i_{1}} \widehat K_{i_{1}}) t_{1}}\right\|  \diff \mu_{i_{1}}(t_{1})
\dotsm\diff \mu_{i_{r}}(t_{r}). \label{EstimMaxLyap}\end{align}
Since $\sum_{i=1}^{N} P_{i}^{{\mathrm{T}}}P_{i} = \id_{d}$ and $P_{i}
e^{(\widehat A + \widehat B_{i} \widehat K_{i}) t} P_{j}^{{\mathrm{T}}}= 0$ if
$i \not = j$, we have, for every $(i_{1}, \dotsc, i_{r}) \in\underline N^{r}$
and $(t_{1}, \dotsc, t_{r}) \in\mathbb{R}_{+}^{r}$,
\begin{align}
&  e^{(\widehat A + \widehat B_{i_{r}} \widehat K_{i_{r}}) t_{r}} \dotsm
e^{(\widehat A + \widehat B_{i_{1}} \widehat K_{i_{1}}) t_{1}}
\displaybreak[0]\nonumber\\
{} = {}  &  \left(  \sum_{j_{r} = 1}^{N} P_{j_{r}}^{{\mathrm{T}}}P_{j_{r}}\right)  e^{(\widehat A + \widehat B_{i_{r}} \widehat K_{i_{r}}) t_{r}}
\dotsm\left(  \sum_{j_{1} = 1}^{N} P_{j_{1}}^{{\mathrm{T}}}P_{j_{1}}\right)
e^{(\widehat A + \widehat B_{i_{1}} \widehat K_{i_{1}}) t_{1}} \left(
\sum_{j_{0} = 1}^{N} P_{j_{0}}^{{\mathrm{T}}}P_{j_{0}}\right)
\displaybreak[0]\nonumber\\
{} = {}  &  \sum_{i = 1}^{N} P_{i}^{{\mathrm{T}}}P_{i} e^{(\widehat A +
\widehat B_{i_{r}} \widehat K_{i_{r}}) t_{r}} \dotsm P_{i}^{{\mathrm{T}}}P_{i}
e^{(\widehat A + \widehat B_{i_{1}} \widehat K_{i_{1}}) t_{1}} P_{i}^{{\mathrm{T}}}P_{i} \displaybreak[0]\nonumber\\
{} = {}  &  \sum_{i = 1}^{N} P_{i}^{{\mathrm{T}}}e^{(A_{i} + \delta_{i i_{r}}
B_{i} K_{i}) t_{r}} \dotsm e^{(A_{i} + \delta_{i i_{1}} B_{i} K_{i}) t_{1}}
P_{i}. \label{EqProdExponentials}\end{align}
Since, for every $i \in\underline N$ and $t \geq0$, we have $\left\|  e^{A_{i}
t}\right\|  \leq C e^{\beta t}$ and $\left\|  e^{(A_{i} + B_{i} K_{i})
t}\right\|  \leq C \gamma^{D} e^{-\gamma t}$, we get, for every $(i_{1},
\dotsc, i_{r}) \in\underline N^{r}$ and $(t_{1}, \dotsc, t_{r}) \in
\mathbb{R}_{+}^{r}$,
\begin{equation}
\label{LargeBound}\left\|  e^{(\widehat A + \widehat B_{i_{r}} \widehat
K_{i_{r}}) t_{r}} \dotsm e^{(\widehat A + \widehat B_{i_{1}} \widehat
K_{i_{1}}) t_{1}}\right\|  \leq N C^{r} \gamma^{r D} e^{\beta\sum_{i=1}^{r}
t_{i}}.
\end{equation}

When $(i_{1}, \dotsc, i_{r}) = (i_{1}^{\ast}, \dotsc, i_{r}^{\ast})$, we can
obtain a sharper bound than \eqref{LargeBound}. For $i \in\underline N$,
denote by $N(i)$ the nonempty set of all indices $k \in\underline{r}$ such
that $i_{k}^{\ast}= i$, and denote by $n(i) \in\mathbb{N}^{\ast}$ the number
of elements in $N(i)$. Then
\[
\left\|  P_{i}^{{\mathrm{T}}}e^{(A_{i} + \delta_{i i_{r}^{\ast}} B_{i} K_{i})
t_{r}} \dotsm e^{(A_{i} + \delta_{i i_{1}^{\ast}} B_{i} K_{i}) t_{1}}
P_{i}\right\|  \leq C^{r} \gamma^{n(i) D} e^{-\gamma\sum_{k \in N(i)} t_{k}}
e^{\beta\sum_{k \in\underline{r} \setminus N(i)} t_{k}},
\]
which shows, using \eqref{EqProdExponentials}, that
\begin{align}
\left\|  e^{(\widehat A + \widehat B_{i_{r}^{\ast}} \widehat K_{i_{r}^{\ast}})
t_{r}} \dotsm e^{(\widehat A + \widehat B_{i_{1}^{\ast}} \widehat
K_{i_{1}^{\ast}}) t_{1}}\right\|   &  \leq\sum_{i=1}^{N} C^{r} \gamma^{n(i) D}
e^{-\gamma\sum_{k \in N(i)} t_{k}} e^{\beta\sum_{k \in\underline{r} \setminus
N(i)} t_{k}}\nonumber\\
&  \leq N C^{r} \gamma^{r D} e^{-\gamma\min_{k \in\underline{r}} t_{k}} e^{r
\beta\max_{k \in\underline{r}} t_{k}}. \label{SharpBound}\end{align}

Let
\begin{align*}
E_{0}  &  = \max_{i \in\underline N} \tau_{i},\\
E_{\min}  &  = \int_{(0, \infty)^{r}} \min_{k \in\underline{r}} t_{k}
\diff \mu_{i_{1}^{\ast}}(t_{1}) \dotsm\diff \mu_{i_{r}^{\ast}}(t_{r}) > 0,\\
E_{\max}  &  = \int_{(0, \infty)^{r}} \max_{k \in\underline{r}} t_{k}
\diff \mu_{i_{1}^{\ast}}(t_{1}) \dotsm\diff \mu_{i_{r}^{\ast}}(t_{r}) <
\infty.
\end{align*}
Then, combining \eqref{LargeBound} and \eqref{SharpBound}, we obtain from
\eqref{EstimMaxLyap} that
\begin{align}
&  \int_{\Omega}\log\left\|  \mathit{\Phi}(r, \omega)\right\|
\diff \mathbb{P}(\omega) \leq N^{r} \left(  \log(N C^{r}) + r D \log\gamma+ r
\beta E_{0}\right) \nonumber\\
&  \hspace{1.5cm} {} + p_{i_{1}^{\ast}} M_{i_{1}^{\ast}i_{2}^{\ast}} \dotsm
M_{i_{r - 1}^{\ast}i_{r}^{\ast}} \left(  \log(N C^{r}) + r D \log\gamma-
\gamma E_{\min} + r \beta E_{\max}\right)  . \label{FinalEstimMaxLyap}\end{align}
The right-hand side of \eqref{FinalEstimMaxLyap} tends to $-\infty$ as
$\gamma\to\infty$, which can be achieved by \eqref{EstimCGLW}. Hence it
follows from Corollary \ref{CoroLambdaMaxCD} that the maximal Lyapunov
exponent of \eqref{ClosedLoop} can be made arbitrarily small.
\end{proof}

Recall that the main motivation for Theorem \ref{MainTheoApp}
comes from the stabilizability of persistently excited systems
\eqref{IntroPESyst} under linear feedback laws. Let us now provide an
application of Theorem \ref{MainTheoApp} to \eqref{IntroPESyst}. To do so, let
$\mu_{1}, \mu_{2} \in\mathrm{Pr}(\mathbb{R}_{+}^{\ast})$ have finite
expectation and $M \in\mathcal{M}_{2}(\mathbb{R})$ be right-stochastic and
irreducible with unique invariant probability vector $p \in\mathbb{R}^{2}$. We
also slightly modify Definition \ref{DefPmbAlpha} for the remainder of this
section by saying that, for $\omega= (i_{n}, t_{n})_{n=1}^{\infty}$, one has
$\pmb\alpha(\omega)(t) = 2 - i_{n}$ for $t \in[s_{n-1}, s_{n})$ and $n
\in\mathbb{N}^{\ast}$, which amounts to saying that $\pmb\alpha(\omega)$ takes
the value $0$ in the state $i = 2$ and the value $1$ in the state $i = 1$. As
a consequence of Theorem \ref{MainTheoApp}, we obtain the following result for \eqref{IntroPESyst}.

\begin{corollary}
\label{CoroToMainTheo} Let $d, m \in\mathbb{N}$, $A \in\mathcal{M}_{d}(\mathbb{R})$, $B \in\mathcal{M}_{d, m}(\mathbb{R})$, and consider system
\eqref{IntroPESyst}. If $(A, B)$ is controllable, then, for every $\lambda
\in\mathbb{R}$, there exists $K \in\mathcal{M}_{m, d}(\mathbb{R})$ such that
the maximal Lyapunov exponent $\lambda_{\max}^{\mathrm{c}}$ of the closed-loop
random switched system $\dot x(t) = (A + \pmb\alpha(\omega)(t) B K) x(t)$
satisfies $\lambda_{\max}^{\mathrm{c}} \leq\lambda$.
\end{corollary}

\begin{proof}
The corollary follows immediately from Theorem \ref{MainTheoApp} by letting $N
= 2$, $A_{1} = A$, $B_{1} = B$, and adding a trivial second subsystem with
$d_{2} = m_{2} = 0$.
\end{proof}

It was proved in \cite[Proposition 4.5]{Chitour2010Stabilization} that there
are (two dimensional) controllable systems for which the achievable decay
rates under persistently exciting signals through linear feedback laws are
bounded below, even when we consider only persistently exciting signals
$\alpha$ taking values in $\{0, 1\}$ instead of $[0, 1]$.
Corollary \ref{CoroToMainTheo} shows that, in the
probabilistic setting defined above, one can get arbitrarily large (almost
sure) decay rates for \eqref{IntroPESyst}, which is in
contrast to the situation for persistently excited systems. An explanation for
this fact is that the probability of having a signal $\alpha$ with very fast
switching for an infinitely long time, such as the signals used in the proof
of \cite[Proposition 4.5]{Chitour2010Stabilization}, is zero, and hence such
signals do not interfere with the behavior of the (random) maximal Lyapunov exponent.

Notice that, in general, $\pmb\alpha(\omega)$ is not $(T, \mu)$-persistently
exciting, but it can be shown to satisfy a condition similar to \eqref{CondPE}
in an asymptotic sense.

\begin{definition}
Let $\rho> 0$ and $\alpha: \mathbb{R}_{+} \to[0, 1]$ be measurable. We say
that $\alpha$ is \emph{$\rho$-asymptotically persistently exciting} if
\[
\liminf_{t \to\infty} \frac{1}{t} \int_{0}^{t} \alpha(s) \diff s \geq\rho.
\]

\end{definition}

It follows easily from \eqref{CondPE} that every $(T, \mu)$-persistently
exciting signal is $\frac{\mu}{T}$-asymptotically persistently exciting. In
order to prove that the above signals $\pmb\alpha(\omega)$ are almost surely
asymptotically persistently exciting for a suitable constant $\rho> 0$, we
assume, in order to simplify the proof, that, in the probabilistic model of
$\pmb\alpha$, trivial switches do not occur, which amounts to choosing
\begin{equation}
\label{DefiMPE}M =
\begin{pmatrix}
0 & 1\\
1 & 0
\end{pmatrix}
\end{equation}
with its unique invariant probability vector $p=\left(  \frac{1}{2}, \frac
{1}{2}\right)  $.

\begin{proposition}
Let $M$ be given by \eqref{DefiMPE}, $p$ be its unique invariant probability
vector, and $\mu_{1}, \mu_{2} \in\mathrm{Pr}(\mathbb{R}_{+}^{\ast})$ have
finite expectations $\tau_{1}, \tau_{2} \in\mathbb{R}_{+}^{\ast}$, respectively.

\begin{enumerate}
\item \label{PmbAlphaNotPE} If $\mu_{2}((0, t]) < 1$ for every $t > 0$, then,
for almost every $\omega\in\Omega$, the signal $\pmb\alpha(\omega)$ is not
$(T, \mu)$-persistently exciting for any $T, \mu\in\mathbb{R}_{+}^{\ast}$ with
$T \geq\mu$.

\item \label{PmbAlphaAsymptPE} For almost every $\omega\in\Omega$, the signal
$\pmb\alpha(\omega)$ is $\frac{\tau_{1}}{\tau_{1} + \tau_{2}}$-asymptotically
persistently exciting.
\end{enumerate}
\end{proposition}

\begin{proof}
To prove \ref{PmbAlphaNotPE}, we show that
\begin{equation}
\mathbb{P}\{\omega\in\Omega\;|\;\exists T\geq\mu>0\text{ such that }\pmb\alpha(\omega)\text{ is a PE }(T,\mu)\text{-signal}\}=0.\label{PEImpossible}\end{equation}
Since a $(T,\mu)$-signal is also a $(T^{\prime},\mu^{\prime})$-signal for
every $T^{\prime}\geq T$ and $0<\mu^{\prime}\leq\mu$, we have
\begin{align*}
&  \{\omega\in\Omega\;|\;\exists T\geq\mu>0\text{ such that }\pmb\alpha
(\omega)\text{ is a PE }(T,\mu)\text{-signal}\}\\
{}={} &  \bigcup_{T>0}\bigcup_{\mu\in(0,T]}\{\omega\in\Omega\;|\;\pmb\alpha
(\omega)\text{ is a PE }(T,\mu)\text{-signal}\}\\
{}={} &  \bigcup_{T\in\mathbb{N}^{\ast}}\bigcup_{\frac{1}{\mu}\in
\mathbb{N}^{\ast}}\{\omega\in\Omega\;|\;\pmb\alpha(\omega)\text{ is a PE
}(T,\mu)\text{-signal}\}.
\end{align*}
If $\alpha$ is a PE $(T,\mu)$-signal, the PE condition implies that $\alpha$
cannot remain zero during time intervals longer than $T-\mu$, and thus
\begin{align}
&  \{\omega\in\Omega\;|\;\pmb\alpha(\omega)\text{ is a PE }(T,\mu
)\text{-signal}\}\nonumber\\
{}\subset{} &  \{\omega=(i_{n},t_{n})_{n=1}^{\infty}\in\Omega\;|\;\forall
n\in\mathbb{N}^{\ast}:\;i_{n}=2\implies t_{n}\leq T-\mu
\}.\label{PEImpliesDoesNotRemainInZero}\end{align}
Since $i_{n}$ takes the value $2$ infinitely many times for almost every
$\omega\in\Omega$ and $\mu_{2}((0,T-\mu])<1$, the right-hand side of
\eqref{PEImpliesDoesNotRemainInZero} has measure zero, and thus
\eqref{PEImpossible} holds.

Proposition \ref{PropProportionTime} implies that, for almost every $\omega
\in\Omega$,
\[
\lim_{t \to\infty} \frac{1}{t} \int_{0}^{t} \pmb\alpha(\omega)(s) \diff s =
\frac{\tau_{1}}{\tau_{1} + \tau_{2}},
\]
and thus \ref{PmbAlphaAsymptPE} holds.
\end{proof}

\bibliographystyle{abbrv}
\bibliography{Bib}

\end{document}